\documentclass[graybox,envcountsame,envcountsect]{svmult}

\usepackage{type1cm}        % activate if the above 3 fonts are
\usepackage{graphicx}        % standard LaTeX graphics tool
\usepackage{newtxtext}       %
\usepackage{newtxmath}       % selects Times Roman as basic font

\usepackage[T1]{fontenc}
\usepackage[utf8]{inputenc}

\usepackage{float}
\usepackage{tikz-cd}
\usepackage{tikz}
\usepackage{pictex}

\usetikzlibrary{decorations.markings}
\tikzstyle{vertex}=[circle, draw, inner sep=0pt, minimum size=6pt]

\newcommand{\vertex}{\node[vertex]}

\usepackage[pagebackref,bookmarksopen]{hyperref}
\hypersetup{
	colorlinks=true,
	linkcolor=blue,
	citecolor=magenta,
	urlcolor=cyan,
}
\def\A{{\mathbb A}}
\def\C{{\mathbb C}}

\def\F{{\mathbb F}}
\def\K{{\mathbb K}}
\def\P{{\mathbb P}}
\def\Q{{\mathbb Q}}
\def\R{{\mathbb R}}
\def\Z{{\mathbb Z}}

\def\ord{\operatorname{ord}}
\def\ac{\operatorname{ac}}
\def\top{\operatorname{top}}
\def\mot{\operatorname{mot}}
\def\div{\operatorname{div}}
\def\Jac{\operatorname{Jac}}
\def\Pic{\operatorname{Pic}}
\def\Exp{\operatorname{Exp}}
\def\lct{\operatorname{lct}}
\def\lcm{\operatorname{lcm}}

\begin{document}

\title*{Introduction to the monodromy conjecture}
\author{Willem Veys}
\institute{Willem Veys \at KU Leuven, Departement Wiskunde, Celestijnenlaan 200B, 3001 Leuven, \email{wim.veys@kuleuven.be}}
%\and Name of Second Author \at Name, Address of Institute \email{name@email.address}}
\maketitle

\abstract*{The monodromy conjecture is a mysterious open problem in singularity theory. Its original version relates arithmetic and topological/geometric properties of a multivariate polynomial $f$ over the integers, more precisely, poles of the $p$-adic Igusa zeta function of $f$ should induce monodromy eigenvalues of $f$. The case of interest is when the zero set of $f$ has singular points.
We first present some history and motivation. Then we expose a proof in the case of two variables, and partial results in higher dimension, together with geometric theorems of independent interest inspired by the conjecture. We conclude with various generalizations.
}

\abstract{The monodromy conjecture is a mysterious open problem in singularity theory. Its original version relates arithmetic and topological/geometric properties of a multivariate polynomial $f$ over the integers, more precisely, poles of the $p$-adic Igusa zeta function of $f$ should induce monodromy eigenvalues of $f$. The case of interest is when the zero set of $f$ has singular points.
We first present some history and motivation. Then we expose a proof in the case of two variables, and partial results in higher dimension, together with geometric theorems of independent interest inspired by the conjecture. We conclude with several possible generalizations.}

%%%%%%% Begining of text
\section*{Contents}
\setcounter{minitocdepth}{3}
\dominitoc

%\newpage

\section{Introduction}

The monodromy conjecture is an intriguing and mysterious problem in singularity theory, on the crossroads of number theory, algebra,  analysis, geometry and topology.
Its original statement predicts a remarkable relation between arithmetic and geometric/topological properties of a polynomial with integer coefficients.
More precisely, let $f\in \Z[x_1,\dots,x_n]$ and  $p$ a prime number. What can we say about the numbers of solutions of $f=0$ over the finite residue rings $\Z/p^i\Z$ for varying $i$?
These data turn out to be equivalent to the knowledge of a certain $p$-adic integral involving a complex parameter $s$, nowadays called the ($p$-adic) Igusa zeta function $Z_p(f;s)$. Igusa showed that it is in fact a rational function in $p^{-s}$. In particular, its poles govern the asymptotic behaviour of the numbers of solutions above when $i$ tends to infinity.

The conjecture predicts that these poles induce monodromy eigenvalues of $f$: when $s_0\in\C$ is a pole of $Z_p(f;s)$, then $e^{2\pi i \Re(s_0)}$ should be a monodromy eigenvalue of $f$ at some complex point of $\{f=0\}$. These eigenvalues are certain invariants of $f$ in the context of differential topology, where $f$ is viewed as a function from $\C^n$ to $\C$.
A priori there seems to be no reason to expect such an implication; it is somewhat motivated by a similar statement, which is a theorem for  analogous real or complex integrals, nowadays called \lq archimedean zeta functions\rq.
When the hypersurface $\{f=0\}$ has no singularities, both its Igusa zeta function and monodromy eigenvalues are quite simple and the conjecture is obvious. It is really a problem within different aspects of singularity theory.

Later Denef and Loeser constructed two related, more geometric, \lq zeta functions of Igusa type\rq: the  topological and motivic zeta function, inspired by a formula for $Z_p(f;s)$ in terms of an embedded resolution of $\{f=0\}$ and by the context of motivic integration, respectively.  Similar monodromy conjectures can be stated for these zeta functions.

Somewhat frustratingly, a conceptual reason for the conjecture to hold is still missing. All proofs of all known cases are ad hoc; one essentially derives enough information on both sides (poles and eigenvalues), often using resolution of singularities and requiring hard work, and then the conclusion follows by inspection of that information.
On the other hand,  attempts of proofs or effective proofs, as well as quests for counterexamples, regularly produce geometric results of independent interest.

\smallskip
We sketch the \lq archimedean history\rq, leading to the original $p$-adic formulation of the conjecture, and introduce the other zeta functions of Igusa type.
We explain a proof for polynomials in two variables (the original proof is due to Loeser in 1988), which is up to now the only dimension where the conjecture is known to hold in general, and we present some significant results for polynomials in three variables, focussing for simplicity of notation on the topological zeta function.  Here some geometric results of independent interest pop up.
We summarize various special cases when the conjecture was proven, and we conclude with possible generalizations of the conjecture in different directions.

Another introductory text with focus on the motivic zeta function is \cite{Nicaise}.
The survey paper \cite{VS} is a substantial introduction to $p$-adic and motivic integration, featuring also zeta functions of Igusa type and applications to singularities.

%\bigskip
%We fix notation and terminology concerning hypersurface singularities. Let $K$ be any field, with algebraic closure $\bar K$, and $f \in K[x_1, \dots, x_n] \setminus K$. We denote the zero locus of $f$ by $\{f=0\}$.

%The singular locus of $\{f= 0 \}$ consists of all  $a\in {\bar K}^n$ satisfying  $f(a)=0$ and  $\partial f/\partial x_i (a) =0$ for all $i=1,\dots, n$; that is, we consider here $f$ with nonreduced structure.
%Otherwise we call $a$ a smooth or nonsingular point of $\{f= 0 \}$.
%We say that $a$ is an isolated singularity of $\{f= 0 \}$ if $\{a\}$ is a connected component of the singular locus of $\{f= 0 \}$.

\section{History}
\label{sec:1}

The original monodromy conjecture connects number theory and geometry/topology; however, it arose from a problem in analysis, posed at the International Congress of Mathematicians in 1954.

\subsection{Archimedean zeta functions}\label{archimedean}
\label{subsec.1.1}

We start with a basic example from calculus. Consider the real integral
$$
I(s) = \int_0^\infty x^s \varphi(x) dx,
$$
where $\varphi$ is  a complex $C^\infty$ function with compact support  and $s$ is a  complex parameter. One easily verifies that $I(s)$ converges for $s\in \C, \Re(s) >-1$, and is holomorphic in that region of the complex plane. In order to investigate its possible continuation, we decompose $I(s)$ as
$$
 \int_0^1  x^s\big(\varphi(x)-\varphi(0)\big) dx  + \int_0^1 x^s \varphi(0) dx  + \int_1^\infty x^s \varphi(x) dx \, .
$$
Since $\varphi(x)-\varphi(0) = x\tilde \varphi(x)$ for another such function $\tilde \varphi$, the first summand converges for $\Re(s) >-2$. The second summand equals $\frac{\varphi(0)}{s+1}$, and the third one converges and is holomorphic on $\C$.  Hence $I(s)$ can be meromorphically continued to the region where $\Re(s) >-2$, with $-1$ as possible pole.
Repeating this procedure yields that $I(s)$ has a meromorphic continuation to $\C$ with possible (simple) poles in $\Z_{<0}$.

\smallskip
What happens if we generalize the univariate polynomial $x$ to a general multivariate  $f \in \R[x_1, \dots, x_n] \setminus \R$? We consider again a so-called \lq test function\index{test function}\rq\ $\varphi:\R^n\to \C$, i.e., a $C^\infty$ function with compact support, and we associate to these data the (real) {\em zeta function\index{real zeta function}}
$$
Z(s) = Z(f,\varphi;s) := \int_{\R^n} |f(x)|^s \varphi(x) dx.
$$
Strictly speaking, the integrand is not defined on $\{f=0\}$, but we will always neglect such subsets of measure zero, in order to simplify notation. (These integrals were sometimes called \lq complex powers\index{complex power}\rq\ and denoted as $f^s$, and they were rather considered as distributions on the space of test functions, see e.g. \cite{Igusa4}.)
Now one easily verifies that $Z(s)$ converges for $s\in \C, \Re(s) >0$, and is holomorphic in that region of the complex plane.  During a talk at the ICM in 1954, I. Gel'fand posed the meromorphic continuation of
$Z(s)$ as an open problem, motivated by applications to the existence of fundamental solutions of linear PDE's with constant coefficients.  Around 1970, Bernstein and S. Gel'fand \cite{Bernstein-Gelfand} and Atiyah \cite{Atiyah} proved this using resolution of singularities, and Bernstein gave another proof \cite{Bernstein2},
using the so-called b-function\index{b-function} (note that both techniques were not available in 1954).
We explain the main ideas in both proofs; for a more elaborate overview on this problem, we refer to \cite{Le}, and for all details to \cite[Chapter 5]{Igusa4}.  We will see that each proof provides also a complete list of possible candidate poles.

\medskip
\noindent {\sc First proof.}
Let $h:Y \to \R^n$ be an embedded resolution\index{embedded resolution}\index{resolution of singularities} of $\{f=0\}$.  So $Y$ is an $n$-dimensional (real) manifold, $h$ is proper, the restriction of $h$ to $Y\setminus h^{-1}\{f=0\}$ is an isomorphism and $h^{-1}\{f=0\} = \cup_{j\in S}E_j$ is a simple normal crossings\index{simple normal crossings divisor} divisor, that is, its irreducible components $E_j$ are nonsingular hypersurfaces that intersect transversally.
 (See e.g. \cite{Hironaka} and more specifically \cite{BEV}\cite{EncinasNobileVillamayor}\cite{Wl2}\cite{Wl3} for such a construction.)

To each $E_j$ one associates its {\em numerical data\index{numerical data}} $(N_j, \nu_j)$, defined (locally) as follows.
In local coordinates $y=(y_1,\dots,y_n)$ on a small enough chart $U_P$ around a point $P \in Y$, one can write $f\circ h$ and $\Jac_h$ (the Jacobian determinant of $h$) in the form
\begin{equation}\label{numerical data in local coordinates}
(f\circ h)(y) = u_P(y) \prod_{i\in S_P} y_i^{N_i} \quad\text{and}\quad \Jac_h(y)= v_P(y)\prod_{i\in S_P} y_i^{\nu_i -1},
\end{equation}
respectively, where $P\in E_i$ precisely for $i\in S_P\subset S$, and $u_P$ and $v_P$ are invertible on $U_P$.
Globally this means that the divisors of zeroes of $f\circ h$ and of $h^*(dx_1\wedge\dots\wedge dx_n)$ are precisely
$$
\sum_{j\in S} N_j E_j  \quad  \text{ and }\quad  \sum_{j\in S} (\nu_j -1) E_j,
$$
respectively.
Using $h$ as a change of variables, we have that
$$
Z(s) = \int_Y |(f\circ h)(y)|^s |\Jac_h(y)| (\varphi \circ h)(y) dy.
$$
Since $h$ is proper, also $\varphi \circ h$ has compact support, and we need only finitely many charts $U_P$ to cover this support.  With notation (\ref{numerical data in local coordinates}) on these charts, and a
partition of the unity $\{\rho_P\}_P$ subordinated to this finite cover, we can write $Z(s)$ as a finite sum
$$\sum_{P} \int_{U_P} \rho_P(y) |u_P(y)|^s |v_P(y)|  (\varphi \circ h)(y)  \prod_ {i\in S_P}  |y_i|^{N_i s+ \nu_i -1} dy_i.
$$
Such \lq monomial\rq\ integrals are easy to handle $n$-dimensional generalizations of the basic example above, see e.g. \cite[Lemma 7.3]{AGV}, and we obtain the following.

\begin{theorem}\label{poles archimedean}
Let $f \in \R [x_1, \dots, x_n] \setminus \R$ and $\varphi: \R^n \to \C$ a $C^\infty$ function with compact support. Then  $Z(f,\varphi;s)$  has a meromorphic continuation to $\C$.
Moreover, if $h$ is an embedded resolution of $\{f=0\}$ with numerical data $(N_j, \nu_j), j\in S$, then the poles of $Z(f,\varphi;s)$ are of the form
$ -\frac{\nu_j +k}{N_j}$,  with $ j\in S$ and  $ k\in \Z_{\geq 0} $.
\end{theorem}

\smallskip
\noindent {\sc Second proof.}
We first briefly introduce the b-function.  For details and more information, we refer to e.g. \cite[Chapter 4]{Igusa4}\cite{Bernstein}.   Let $K$ be a field of characteristic zero and $n$ a positive integer. The Weyl algebra $\mathcal{D}_n$ is the (non-commutative) algebra of algebraic differential operators $K\langle x,\partial/\partial x\rangle := K\langle x_1\dots,x_n,\partial/\partial x_1,\dots,\partial/\partial x_n \rangle$, in the obvious way acting on $K[x_1\dots,x_n]$.  Adding another variable $s$, we obtain the algebra $\mathcal{D}_n[s]$, where $s$ acts on $K[x_1\dots,x_n]$ simply by multiplication.

\begin{theorem}[\cite{Bernstein}] Let $f\in K[x_1\dots,x_n]\setminus K$. There exists a nonzero operator $P(s)\in \mathcal{D}_n[s]$ and a nonzero polynomial $b(s)\in K[s]$ satisfying
\begin{equation}\label{bernstein}
P(s) \cdot f^{s+1} = b(s) f^s.
\end{equation}
\end{theorem}
This identity can be viewed as either \lq formal\rq\ or for  integer values of $s$.  The unique monic polynomial of smallest degree satisfying (\ref{bernstein}) is called the
{\em Bernstein-Sato polynomial\index{Bernstein-Sato polynomial}} or {\em b-function\index{b-function}} of $f$, denoted $b_f(s)$.
For $a\in K^n$, there is also a local version $b_{f,a}(s)$, %\cite{Bj}
and  $b_f= \lcm_{a\in K^n} b_{f,a}$; see e.g. \cite[Proposition 4.2.1]{MN}.

For ease of exposition, we assume now that  $f(x) \geq 0$ on the domain of $\varphi$. (Otherwise one splits $f$ in a positive and negative part and applies the following argument twice.)
Using (\ref{bernstein}) for $b_f$, we compute
$$
\begin{aligned}
b_f(s) Z(s) &=  \int_{\R^n} b_f(s)  f(x)^s \varphi(x) dx =  \int_{\R^n} P(x, \partial /\partial x,s) \cdot f(x)^{s+1} \varphi(x) dx \\
&= \int_{\R^n}  f(x)^{s+1} P^*(x, \partial /\partial x,s) \cdot \varphi(x) dx,
\end{aligned}
$$
where the last equality is an instance of integration by parts, via the adjoint operator $P^*$ of $P$.
Hence
$$
Z(s) = \frac 1{b_f(s)}  \int_{\R^n}  f(x)^{s+1} \varphi_1(x,s) dx,
$$
for some other test function $\varphi_1$, depending polynomially on $s$.  We can continue this process, and obtain for arbitrary $r\in \Z_{\geq 1}$ that
$$
Z(s) = \frac 1{b_f(s+r-1) \dots b_f(s+1)b_f(s)}  \int_{\R^n}  f(x)^{s+r} \varphi_r(x,s) dx,
$$
for some test function $\varphi_r$, and conclude finally the following.

\begin{theorem}\label{poles archimedean2}
Let $f \in \R [x_1, \dots, x_n] \setminus \R$ and $\varphi: \R^n \to \C$ a $C^\infty$ function with compact support. Then  $Z(f,\varphi;s)$  has a meromorphic continuation to $\C$.
Moreover, its poles  are of the form
$ \lambda -k$,  with $\lambda$ a root of $b_f$ and  $ k\in \Z_{\geq 0} $.
\end{theorem}
Note that this list is consistent with Theorem \ref{poles archimedean}, see Theorem \ref{roots b}.

A somewhat finer analysis, see e.g. \cite[Theorem 5.3.1]{Igusa4}, yields the following. Say $b_f$ factors as $b_f(s)=\prod_\mu (s+\mu)$. Then  $Z(f,\varphi;s)/ \prod_\mu \Gamma(s+\mu)$ has a holomorphic continuation to $\C$, where $\Gamma(\cdot)$ is the gamma function.

\begin{remark}\label{complex and Schwartz} (1) One has complete analogues of Theorems \ref{poles archimedean} and \ref{poles archimedean2} for complex zeta functions\index{complex zeta function}
$ Z(s) = Z(f,\varphi;s) := \int_{\C^n} |f(x)|^{2s} \varphi(x) dxd\bar x $,
where now $f \in \C [x_1, \dots, x_n] \setminus \C$. We refer again to \cite{Le} and  \cite{Igusa4} for details. {\em Archimedean zeta functions\index{archimedean zeta function}} is a  collective name for the real and complex zeta functions.

(2) These theorems can in the two archimedean settings be generalized to {\em Schwartz-Bruhat functions\index{Schwartz-Bruhat function}} $\varphi$. These are $C^\infty$ functions on $\R^n$ or $\C^n$ with \lq very fast decay at infinity\rq. More precisely, $\varphi$ and all its partial derivatives are globally bounded, that is,
the supremum over all $x$ in $\R^n$ or $\C^n$ of $|P\varphi(x)|$ is finite, for every $P\in \mathcal{D}_n$. See again  \cite{Igusa4} for more information.
\end{remark}

\subsection{Relation with monodromy}\label{relation monodromy}

We first make terminology precise. Let $f \in \C[x_1, \dots, x_n] \setminus \C$ and $a\in \{f=0\}\subset \C^n$. We say that $a$ is a singular point of $\{f= 0 \}$ if  $\partial f/\partial x_i (a) =0$ for all $i=1,\dots, n$; that is, somewhat abusing notation, we do take into account that  $f$ is possibly non-reduced.
%  Otherwise we call $a$ a smooth or nonsingular point of $\{f= 0 \}$.
We say that $a$ is an isolated singularity of $\{f= 0 \}$ if $a$ is a singular point of $\{f= 0 \}$, and moreover we have in a small neighbourhood of $a$ that all points of $\{f= 0 \}$ different from $a$ are smooth points of $\{f= 0 \}$.

\bigskip
Consider  $f \in \C[x_1, \dots, x_n] \setminus \C$ as a map $\C^n\to \C$ and fix $a\in \{f=0\} \subset \C^n$. We briefly define the Milnor fibre of $f$ at $a$ and its monodromy, as introduced by Milnor in \cite{Milnor}, where all details can be found.
Choose a small enough ball $B\subset \C^n$ with centre $a$; then $f|_{B\cap f^{-1}D^*}$ is a locally trivial $C^\infty$ fibration over a small enough pointed disc $D^*\subset \C\setminus 0\}$ with centre $0$. Choose now also $t\in D^*$; the diffeomorphism type (up to isotopy) of \lq the\rq\ {\em (local) Milnor fibre\index{Milnor fibre} $F_a:= f^{-1}(t)\cap B$ of $f$ at $a$} does not depend on $t$, and the counterclockwise generator of the fundamental group of $D^*$ induces linear automorphisms $M_i$ of the cohomology vector spaces $H^i(F_a,\C)$, called (algebraic) monodromy action\index{monodromy action} of $f$ at $a$. By a {\em monodromy eigenvalue\index{monodromy eigenvalue} of $f$ at $a$}
we mean an eigenvalue of  a least one of the $M_i$.
For surveys on this Milnor fibration and associated invariants, we refer to \cite{Bu1}\cite{LNS}\cite{Se}.

 The following properties are well known, and are either easy to see or shown in e.g. \cite{SGA7}\cite{Milnor}.

\begin{proposition}\label{monodromy properties} Let  $f \in \C[x_1, \dots, x_n] \setminus \C$. For $a\in \{f=0\}$, let $P_i(t)$ denote the characteristic polynomial of the monodromy automorphism $M_i$ on $H^i(F_a,\C)$.

(1) All monodromy eigenvalues are roots of unity;

(2)  $H^i(F_a,\C) =0$ for $i \geq  n$.

(3) If $f = \prod_j f_j^{N_j}$ is the  decomposition of $f$ in irreducible components and $d := \gcd_{j, a\in\{f_j=0\}}N_j$, then $P_0(t) = t^d - 1$.

(4) When $a$ is an isolated singularity of $\{f= 0 \}$, then $H^i(F_a,\C) = 0$ for $i \neq 0, n-1$. Moreover, $H^{n-1}(F_a,\C) \neq 0$ and $P_0(t) = t-1$.

(5) When $a$ is a smooth point of $\{f= 0 \}$, then $H^i(F_a,\C) = 0$ for $i>0$, and $P_0(t) = t-1$.
\end{proposition}

Kashiwara and Malgrange established a close relation between roots of the b-function and monodromy eigenvalues. In particular, these eigenvalues are completely determined by the b-function.

\begin{theorem}\label{Malgrange}\cite{Ka}\cite{Ma2} Let $f \in \C[x_1, \dots, x_n] \setminus \C$.  If $s_0$ is a root of $b_f$, then $e^{2\pi i s_0}$ is a monodromy eigenvalue of $f$ at some point of $\{f=0\}$. Conversely, each such eigenvalue is obtained this way.
\end{theorem}

There is also a local variant of this result. Combining it with Theorem \ref{poles archimedean2} yields the following relation between poles and eigenvalues.

\begin{corollary}\label{poles induce eigenvalues}
Let $f \in \R [x_1, \dots, x_n] \setminus \R$ and $\varphi: \R^n \to \C$ a $C^\infty$ function with compact support.
If $s_0$ is a pole of $Z(f,\varphi;s)$, then $e^{2\pi i s_0}$ is a monodromy eigenvalue of $f$ at some point of $\{f=0\}$.
\end{corollary}

\subsection{$p$-adic zeta functions}

We introduce briefly the $p$-adic  analogue of the archimedean zeta functions. We refer to \cite{PV} for an elaborate introduction, and to \cite{DenefBourbaki} and \cite{Meuser} for substantial survey papers.

We  replace $\R$ by the field of $p$-adic numbers\index{$p$-adic numbers} $\Q_p$.  We denote by $\mathrm{ord}_p(\cdot)$ and $\vert \cdot\vert_p=p^{-\mathrm{ord}_p(\cdot)}$ the $p$-order
 and the standard $p$-adic norm on $\Q_p$, respectively.
Recall that the ring of $p$-adic integers  $\Z_p =\{x\in \Q_p \mid  \mathrm{ord}_p(x)\geq 0\}$ is a discrete valuation ring with maximal ideal $p\Z_p=\{x\in \Q_p \mid  \mathrm{ord}_p(x) > 0\}$ and unit group $\Z_p^\times =\{x\in \Q_p \mid  \mathrm{ord}_p(x)= 0\}$.
Every $a\in \Q_p \setminus \{0\}$ can be written in a unique way  as $a=up^z$ with  $u \in \Z_p^\times$ and $z\in \Z$,  where then $\mathrm{ord}_p(a)=z$.

\begin{definition}
Let $f \in \Q_p [x_1, \dots, x_n] \setminus \Q_p$ and $\varphi$ a test function\index{test function} on  $\Q_p^n$, meaning now a locally constant function with compact support. The {\em $p$-adic Igusa zeta function\index{$p$-adic zeta function}\index{Igusa zeta function}\index{$p$-adic Igusa zeta function}} associated to $f$ and $\varphi$ is
$$
Z_p(s)=Z_p(f,\varphi;s) := \int_{\Q_p^n} |f(x)|_p^s \varphi(x) dx ,
$$
where $dx$ denotes the standard Haar measure on $\Q_p^n$ (such that $\Z_p^n$ has measure $1$), and $s\in \C$. Note that $|f(x)|_p^s = (p^{-s})^{\mathrm{ord}_p f(x)}$, implying that
$Z_p(s)$ is in fact a function in $p^{-s}$.

When $\varphi$ is the characteristic function of $\Z_p^n$ and $(p\Z_p)^n$, we denote $Z_p(f,\varphi;s)$ by $Z(f;s)$ and $Z_0(f;s)$, respectively.
\end{definition}

\begin{remark}
At first sight the test functions $\varphi$ above may seem quite different from those in the archimedean setting. Note first that, since $\Q_p^n$ is totally disconnected, there is no reasonable notion of derivative of a function $\Q_p^n  \to \C$.  Second, both the Schartz-Bruhat functions on $\R^n$ from Remark \ref{complex and Schwartz}(2) and  the locally constant functions with compact support on $\Q_p^n$ are both instances of the notion \lq Schwartz space\index{Schwartz space}\rq\ of functions $\K^n \to \C$ for a local field $\K$.  (For example a common feature is that both instances are invariant under Fourier transform.)
\end{remark}

\begin{remark}
The terminology in the literature is sometimes confusing. One uses the adjectives {\em $p$-adic}, {\em Igusa} and/or {\em local} in all possible combinations to indicate the zeta function $Z_p(f,\varphi;s)$. Here \lq local\rq\ refers to the fact that $\Q_p$ is a local field.

However, when it is clear that one works in the $p$-adic setting, the term  \lq local\rq\ sometimes refers to the fact that the integration domain is a smaller neighbourhood of the origin, typically meaning $Z_0(f;s)$, or, more generally, $Z_p(f,\varphi;s)$ with $\varphi$ the characteristic function of $(p^e\Z_p)^n$ for some $e\geq 1$.
\end{remark}

As in the real setting, one easily verifies that $Z_p(s)$ converges and is holomorphic in the region $\Re(s)>0$. What about meromorphic continuation? In particular, are there analogues of the two proofs in the archimedean setting?
\par
\smallskip
(1)  The proof using embedded resolution can readily be adapted to the $p$-adic setting.  (An extra specific reference for resolution in this setting is \cite[Theorem 2.2]{DvdD}.)

Now $h:Y \to \Q_p^n$ is an embedded resolution\index{embedded resolution} of $\{f=0\}$, with $Y$ an $n$-dimensional $p$-adic manifold, and the $E_j,j\in S,$ are again the irreducible components of $h^{-1}\{f=0\}$ with numerical data  $(N_j, \nu_j)$.
Since a compact $p$-adic manifold can be covered by finitely many {\em disjoint} (simultaneously open and closed) balls, a partition of the unity is not needed here. This time  $Z_p(s)$ turns out to be essentially  a finite sum of products of elementary integrals of the form
$$
\int_{p^e\Z_p}\vert y_j \vert_p^{N_js+\nu_j-1} dy_j   = \frac{(1-p^{-1})p^{-e(N_js+\nu_j)}}{1-p^{-(N_js+\nu_j)}},
$$
where the computation yielding the right hand side is an easy exercise, assuming that the integral converges, which is true when $\Re(s)>-\frac{\nu_j}{N_j}$.  For more details we refer to \cite{PV} or \cite{Igusa4}. We can conclude as follows.

\begin{theorem}\label{zetafunction is rational}
Let $f \in \Q_p [x_1, \dots, x_n] \setminus \Q_p$ and $\varphi: \Q_p^n \to \C$ a locally constant function with compact support. Then  $Z_p(f,\varphi;s)$ is a {\em rational function} in $p^{-s}$,  and it has a meromorphic continuation to $\C$ with poles contained in the locus where $p^{-s}=p^{\nu_j/N_j}$ for some $j\in S$.  Hence, in terms of $s$, the poles are of the form
$$
 -\frac{\nu_j}{N_j} + \frac{2\pi k}{(\ln p)N_j}i, \text{ with } \ j\in S \text{ and } k\in \Z \, .
$$
\end{theorem}
Comparing with the real case, starting with the basic values $-\frac{\nu_j}{N_j}$, we have \lq up and down\rq\ imaginary shifts instead of \lq left\rq\ real shifts.

\par
\smallskip
(2)  On the other hand, the proof via the Bernstein-Sato polynomial\index{Bernstein-Sato polynomial} and integration by parts completely breaks down in the $p$-adic setting.
In particular, there seems to be no link between poles of the $p$-adic Igusa zeta function and roots of $b_f$ or monodromy eigenvalues of $f$.
However, Igusa computed several complicated examples where he verified that the real part $\Re(s_0)$ of any pole $s_0$ of $Z_p(f,\varphi;s)$ is a root of $b_f$, and hence that also $e^{2\pi i \Re(s_0)}$ is a monodromy eigenvalue of $f$ \cite{Ig}, and he formulated such a general implication as a conjecture.

\subsection{The conjecture; $p$-adic version}
\begin{conjecture}[Monodromy conjecture]\index{monodromy conjecture}
Let $f\in \Q[x_1,\dots,x_n]\setminus \Q$. Then, for all but a finite number of $p$ and for all locally constant functions with compact support $\varphi: \Q_p^n \to \C$ we have the following. If $s_0$ is a pole of $Z_p(f,\varphi;s)$, then
\begin{enumerate}
\item[(1)] (standard) $e^{2\pi i \Re (s_0)}$ is a monodromy eigenvalue of $f:\C^n\to \C$ at some point of $\{f=0\}$;
\item[(2)] (strong)  $\Re(s_0)$ is a root of $b_f$.
\end{enumerate}
\end{conjecture}
Of course, (2) implies (1) by Theorem \ref{Malgrange}.
There is also a local version: assuming that $f(0)=0$,  real parts of poles of   $Z_0(f;s)$ similarly  induce monodromy eigenvalues of $f$ at points close to $0$, or are roots of $b_{f,0}$.

\begin{remark}
(1) A priori, there is no reason at all to expect these implications, except the possible analogy with the archimedean case.

(2) The implication in the statement might even be true for all primes $p$.

(3) The standard statement would provide a remarkable link between arithmetic and geometric/topological properties of the polynomial $f$, namely poles of $Z_p(f,\varphi;s)$ and monodromy eigenvalues of $f$.  To emphasize this, we mention that for fixed $p$, taking $f\in \Z[x_1,\dots,x_n]\setminus \Z$, knowing  $Z(f;s)$ is equivalent to the knowledge of the numbers of solutions of  $f=0$ over the finite rings $\Z/p^i\Z$  ($i>0$). More precisely, the formal power series
$$
P(T):= 1+\sum_{i>0} \#\{\text{solutions of }  f=0 \text{ over } \Z/p^i\Z\} p^{-ni}T^i
$$
satisfies the identity
$$
Z(f;s) = \frac{(p^{-s}-1)P(p^{-s})+1}{p^{-s}}.
$$
For an elaborate overview, we refer to \cite{PV}.
In particular, the poles of $Z(f;s)$ dictate the asymptotic behaviour of these numbers of solutions as $i$ tends to infinity, see \cite{Segers1} for a precise statement.

(4) The first (and up to now last ...) completely general result on the conjecture is the case of plane curves ($n=2$):  Loeser \cite{Loeser} settled the strong statement in 1982. Rodrigues \cite{Ro} gave a more elementary proof of the standard statement, of which we present a simplified version in \S\ref{sec:3}.
\end{remark}

\begin{example}\label{first example} For $f=(y^2-x^3)^2-x^6y$, we compute \lq both sides\rq\ of the monodromy conjecture.
The origin is the only singular point of the irreducible curve $\{f=0\}$, and its mimimal embedded resolution $h$  is obtained by six blow-ups. Denote by $E_1,E_2,\dots,E_6$ the consecutive exceptional curves and by $E_0$ the strict transform of $\{f=0\}$. In the figure we indicate how they intersect and we decorate each $E_j$ with its numerical data $(N_j,\nu_j)$.

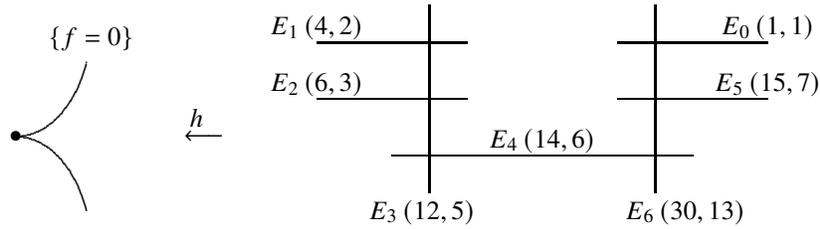
\bigskip
\centerline{
\beginpicture

\setcoordinatesystem units <.5truecm,.5truecm>

\ellipticalarc axes ratio 2:3  70 degrees from -1 0 center at -1 3
\ellipticalarc axes ratio 2:3  -70 degrees from -1 0 center at -1 -3

\put {$\bullet$} at -1 0
 \put {$\{f=0\}$} at 1 2.6

\put {$\longleftarrow$} at 4 0
 \put {$h$} at 3.8 .5

\setcoordinatesystem units <.5truecm,.5truecm> point at -13 0

 \putrule from -4 -0.5  to 4 -0.5
 \putrule from -3  -1.5  to -3 3.5
 \putrule from 3  -1.5  to 3 3.5
 \putrule from 2 2.5 to 6 2.5
 \putrule from 2 1 to 6 1
\putrule from -6 2.5 to -2 2.5
 \putrule from -6 1 to -2 1

 \put {$E_0\, (1,1)$} at 6 2.9
 \put {$E_5\, (15,7)$} at 6 1.4
  \put {$E_3 \, (12,5)$}  at  -3.2  -2
  \put {$E_6 \, (30,13)$}  at  3.8  -2
  \put {$E_1\, (4,2)$} at -6 2.9
  \put {$E_2\, (6,3)$} at -6 1.4
   \put {$E_4\, (14,6)$}  at  0  -.1

\endpicture}

\bigskip
According to Theorem \ref{zetafunction is rational}, the real parts of the candidate poles of both $Z(f;s)$ and $Z_0(f;s)$ are the fractions $ -\nu_j/N_j\, ( 0\leq j \leq 6)$, being the six different values
$$-1, -1/2, -5/12, -3/7, -7/15, -13/30.$$
  However, it turns out that only $-1$, $-5/12$ and $-13/30$ are actually real parts of poles.
%Such a cancellation phenomenon will turn out to be typical.
   More precisely, using Theorem \ref {denefformula}, one could compute
% (for all but finitely many $p$)
an explicit expression for $Z(f;s)$ and $Z_0(f;s)$ of the form
$$ \frac{A(p^s)}{(p^{1+s}-1)(p^{5+12s}-1)(p^{13+30s}-1)}, $$
where $A(p^s)$ is a polynomial in $p^s$.

Next, we compute the monodromy eigenvalues of $f$.  Proposition \ref{monodromy properties} already implies that the cohomology vector spaces $H^i(F_a,\C)$ vanish for all $i\geq 2$ and all $a\in \{f=0\}$, that also $H^1(F_a,\C)$ vanishes for $a\neq 0$, and that  $P_0(t)=t-1$ at any point of  $\{f=0\}$. For the most relevant datum, namely the characteristic polynomial $P_1(t)$ at the origin, there is a classical formula in terms of an embedded resolution of $\{f=0\}$ (in particular using the data $N_j$), due to A'Campo, see Theorem \ref{A'Campo}.  Here the formula yields concretely

$$P_1(t) = \frac{(t^{30}-1)(t^{12}-1)(t-1)}{(t^{15}-1)(t^{6}-1)(t^4-1)} = \Phi_{30}\cdot\Phi_{12}\cdot\Phi_{10}\cdot\Phi_{6},$$
where $\Phi_{\ell}$ is the $\ell$th cyclotomic polynomial.
Hence its $18$ zeroes are all the primitive roots of unity of order $30$, $12$, $10$ and $6$.
As predicted by the monodromy conjecture, the real parts of poles $-1$, $-5/12$ and $-13/30$  indeed induce monodromy eigenvalues.

%\smallskip
The roots of the b-function of $f$ can be computed by a formula due to Yano \cite{Yano2}; they are
\begin{align*}
&-5/12, -13/30, -17/30, -7/12, -19/30, -7/10, -23/30, -5/6, -9/10, -11/12
\\
 & -29/30, -1, -31/30, -13/12, -11/10, -7/6, -37/30, -13/10, -41/30
\end{align*}
(all of multiplicity 1). So, as a finer result, $-1$, $-5/12$ and $-13/30$  are roots of $b_f$.

\bigskip
We observe here several facts that will turn out to be typical or even special cases of theorems.

(1) The fact that the \lq false\rq\ candidates $-\nu_1/N_1=-\nu_2/N_2=-1/2$, $-\nu_4/N_4=-3/7$ and $-\nu_5/N_5=-7/15$ cancel, is related to the fact that the exceptional curves $E_1$, $E_2$, $E_4$ and $E_5$ intersect other components  exactly once or twice.

(2) These candidates $-1/2$, $-3/7$ and $-7/15$ do {\em not} induce monodromy eigenvalues.
So their cancellation  seems strongly related to the monodromy conjecture.

(3) There are a lot more monodromy eigenvalues (as well as roots of the b-function) than real parts of poles.
\end{example}

\section{More \lq Igusa type\rq\ zeta functions and formulas via resolution}
\label{sec:2}

Unlike for the archimedean zeta functions, there is an explicit formula for the $p$-adic zeta function in terms of a given embedded resolution.  That formula leads to a kind of geometric specialization of the $p$-adic zeta function, the {\em topological zeta function}.  On the other hand we have the {\em motivic zeta function}, which is some geometric upgrade of the $p$-adic one. One can formulate an analogous monodromy conjecture for these zeta functions.

From now on we will use repeatedly the algebro-geometric version of embedded resolution, over various fields \cite{Hironaka}. We fix associated notation and terminology for the rest of the paper.

\begin{note}\label{notation embedded resolution}
Let $K$ be a field of characteristic zero and $f\in K[x_1,\dots,x_n]\setminus K$. Let $h:Y\to \A_K^n$ be an embedded resolution\index{embedded resolution}\index{resolution of singularities} of  $\{f=0\}$, defined over $K$.  So $Y$ is an $n$-dimensional smooth algebraic variety, $h^{-1}\{f=0\} = \cup_{j\in S}E_j$ is a simple normal crossings  divisor\index{simple normal crossings divisor}, that is, its irreducible components $E_j$ are nonsingular hypersurfaces that intersect transversally, and the restriction of $h$ to $Y\setminus h^{-1}\{f=0\}$ is an isomorphism.
(One could require that $h$ is moreover an isomorphism outside the inverse image of the singular locus of $\{f=0\}$.)
We write $S=S_e\sqcup S_s$, where $S_e$ ranges over the exceptional components and $S_s$ over the components of the strict transform of $\{f=0\}$.

Moreover, $Y$, $h$ and all $E_j$ are defined over $K$.  To each $E_j$ one associates its {\em numerical data\index{numerical data}} $(N_j, \nu_j)$, where $N_j$ and $\nu_j -1$ are the multiplicities of $E_j$ in the divisor of zeroes of $f\circ h$ and of the $n$-form $h^*(d_1\wedge\dots\wedge dx_n)$, respectively. In other words:
$$
\div(f\circ h)=\sum_{j\in S} N_j E_j  \quad  \text{ and }\quad \div\big(h^*(dx_1\wedge\dots\wedge dx_n)\big)= \sum_{j\in S} (\nu_j -1) E_j.
$$
 Also, for $I \subset S$, we denote $E_I=\cap_{i\in I}E_i$ and $E_I^\circ=E_I\setminus \cup_{k\not\in I}E_k$. Hence $Y$ equals the disjoint union of all $E_I^\circ$.
Note that, in particular, $E_\emptyset^\circ=Y\setminus \cup_{j\in S}E_j$ and that it is isomorphic to $\A_K^n \setminus \{f=0\}$.

To simplify notation, when $I=\{i\}$, we put $E_i:=E_{\{i\}}$ and $E_i^\circ:= E_{\{i\}}^\circ$.
\end{note}

\begin{remark} (1) Denote by $K_{\bullet}$ de canonical divisor\index{canonical divisor} (class) of a normal variety. In fact, $\div(h^*(dx_1\wedge\dots\wedge dx_n))$ equals the {\em relative canonical divisor\index{relative canonical divisor}} $K_h$, which is the unique representative of the divisor class $K_Y-h^*K_{\A^n}$ with support in the exceptional locus of $h$ (which is in this case just $K_Y$ since  $K_{\A^n}$ is trivial).

(2) The numerical data appear in local expressions for $f\circ h$ and for the Jacobian determinant of $h$, exactly as in (\ref{numerical data in local coordinates}).

(3) In the context of singularity germs, we rather consider $\{f=0\}$ only in a subset $U$ of affine space. Then $h$ is map from $Y$ to $U$, for which we use all notations above.
\end{remark}

\begin{remark}
According to e.g. \cite{Hironaka}, one can always construct some embedded resolution  as a composition of blow-ups (with so-called {\em admissible} centre).
For curves ($n=2$), {\em every} embedded resolution is a composition of point blow-ups, and there is a unique minimal embedded resolution, such that every other one is obtained from it by further point blow-ups.  For $n\geq 3$ one needs in general blow-ups with centres of dimension up to $n-2$, and a reasonable notion of minimal embedded resolution, in the sense that all other ones factor through it,  does not exist in general.
\end{remark}

\subsection{Denef's formula }

Consider an embedded resolution $h:Y\rightarrow \Q_p^n$ of $f\in \Q_p[x_1,\dots,x_n]\setminus \Q_p$, where now we consider $\Q_p^n$ as the algebro-geometric affine space   $\A_{\Q_p}^n$, rather than as a $p$-adic manifold. For the subtleties related to this different point of view, we refer to  \cite{PV}. We use notation from Note \ref{notation embedded resolution}.

We will need  the notion of reduction mod $p$\index{reduction mod $p$}: to any algebraic variety or scheme $Z$ over %$\Q$ or
 $\Q_p$, one associates a well-defined variety/scheme over $\mathbb{F}_p$, the {\em reduction} mod $p$ of $Z$,  denoted by $\overline{Z}$.
(For example, when $Z$ is the hypersurface in $\A_{\Q_p}^d$ defined by a nonconstant polynomial $g \in \Q_p[x_1,\dots, x_d]$, there is a unique integer $m$ such that $p^m g \in \Z_p[x_1,\dots, x_d]\setminus p\Z_p[x_1,\dots, x_d]$, and then $\overline{Z}$ is the hypersurface in $\A_{\F_p}^d$ defined by the polynomial obtained from $p^mg$ by reducing all its coefficients mod $p$.)
 Similarly, there is a well-defined notion of reduction mod $p$ for any morphism $g$ of % $\Q$- or
 $\Q_p$-schemes, being a morphism $\overline{g}$ of the corresponding $\mathbb{F}_p$-schemes. See for example \cite{Denef1} and the references therein.

Now for $I \subset S$, we denote similarly  $\overline{E_I}^\circ=\overline{E_I}\setminus \cup_{k\not\in I}\overline{E_k}$.

	\begin{theorem}[Denef's formula \cite{Denef1}]\index{Denef's formula}
		\label{denefformula}
		\noindent
Let $f\in \Q[x_1,\dots,x_n]\setminus\Q$. Let $h:Y\rightarrow \A_{\Q}^n$ be an embedded resolution of $\{f=0\}$, that is defined over $\Q$.

For any prime $p$, we can also consider $f$ in $\Q_p[x_1,\dots,x_n]$ and view $Y$ and $h$ as defined over $\Q_p$ (technically, one performs a base change). Then, for all but finitely many $p$, we have that
\begin{equation}\label{Denef's formula}
			Z(f;s)=\frac{1}{p^n}\sum_{I\subset S}\sharp\big(\overline{E_I}^\circ(\F_p)\big)\prod_{i\in I}\frac{p-1}{p^{\nu_i+N_i s}-1},
\end{equation}
where $\sharp\big(\overline{E_I}^\circ(\F_p)\big)$ is the number of $\F_p$-points in $\overline{E_I}^\circ$.
	\end{theorem}

\begin{remark}\label{generalDenef}
(1) More precisely, for any given $p$, Denef defines when $h$ (considered over $\Q_p$) has {\em good reduction} mod $p$\index{good reduction mod $p$}: this means that $\overline{Y}$ and all $\overline{E_i}$ are smooth, $\cup_{i\in S} \overline{E_i}$ has normal crossings, and $\overline{E_i}$ and $\overline{E_j}$ have no common components for $i\neq j$.
He proves the formula above when $h$ has good reduction mod $p$, and shows that, starting with $f$ and $h$ over $\Q$, the resolution  $h$ has good reduction mod $p$ for all but finitely many $p$.
(Note that for a fixed prime $p$ and a given $f$ over $\Q_p$, it is possible that there exists no embedded resolution of $f$ with good reduction mod $p$.)

(2) There is a more general version for $Z_p(f,\varphi;s)$ if $\varphi$ is the characteristic function of a union of polydiscs $(a_1+p\Z_p,\dots,a_n+p\Z_p) \subset \Z_p^n$ (with $a_1,\dots,a_n \in \{0,1,\dots, p-1\}$); then  $\sharp(\overline{E_I}^\circ(\F_p))$ is replaced in (\ref{Denef's formula}) by the number of $\F_p$-points  in $\overline{E_I}^\circ$ that are mapped by $\overline{h}$ to one of the $(\overline{a_1},\dots, \overline{a_n}) \in \F_p^n$.
 \end{remark}

\begin{example}\label{example} Take $f=y^3-x^5$. The minimal embedded resolution $h:Y\to \A^2$ of $\{f=0\}$ is obtained by four blow-ups. Denote by $E_1,E_2,E_3,E_4$ the consecutive exceptional curves and by $E_0$ the strict transform of $\{f=0\}$. In the figure we indicate how they intersect and we decorate each $E_i$ with its numerical data $(N_i,\nu_i)$.

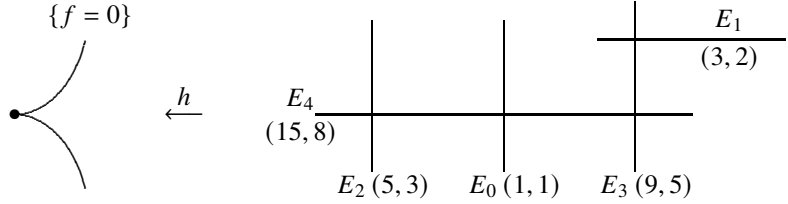
\bigskip
\centerline{
\beginpicture

\setcoordinatesystem units <.5truecm,.5truecm>

\ellipticalarc axes ratio 2:3  70 degrees from 0 0 center at 0 3
\ellipticalarc axes ratio 2:3  -70 degrees from 0 0 center at 0 -3

\put {$\bullet$} at 0 0
 \put {$\{f=0\}$} at 2 2.6

\put {$\longleftarrow$} at 4.5 0
 \put {$h$} at 4.5 .5

\setcoordinatesystem units <.5truecm,.5truecm> point at -13 0

 \putrule from -5 0  to 5 0
 \putrule from  0  -1.5 to 0 2.5
 \putrule from -3.5  -1.5  to -3.5 2.5
 \putrule from 3.5  -1.5  to 3.5 3
 \putrule from 7.5 2 to 2.5 2

 \put {$E_1$} at 6 2.5       \put{$(3,2)$}  at 6 1.5
  \put {$E_2 \, (5,3)$}  at  -3.2  -1.9
  \put {$E_3 \, (9,5)$}  at  3.8  -1.9
  \put {$E_4$} at -5.4  .4     \put{$(15,8)$}  at -5.4  -.5
   \put {$E_0\, (1,1)$}  at  .3  -1.9

\endpicture}

\bigskip

We describe the ten terms in formula (\ref {denefformula}), omitting for simplicity the factor $\frac{1}{p^2}$.  There are four terms with $|I|=2$, corresponding to the intersection points
$E_1\cap E_3$, $E_3\cap E_4$, $E_0\cap E_4$ and $E_2\cap E_4$:
\begin{align*}
&1\frac{(p-1)^2}{(p^{2+3s}-1)(p^{5+9s}-1)}, \qquad  &1\frac{(p-1)^2}{(p^{8+15s}-1)(p^{5+9s}-1)}, \\
&  1\frac{(p-1)^2}{(p^{8+15s}-1)(p^{1+s}-1)}, \qquad & 1\frac{(p-1)^2}{(p^{8+15s}-1)(p^{3+5s}-1)},
\end{align*}
respectively.
There are five terms with $|I|=1$, corresponding to the \lq open parts\rq\ $E_i^\circ$ for $i\in\{0,\dots,4\}$:
$$ (p-1) \frac{p-1}{p^{1+s}-1},  p \frac{p-1}{p^{2+3s}-1},  p \frac{p-1}{p^{3+5s}-1},  (p-1) \frac{p-1}{p^{5+9s}-1},  (p-2) \frac{p-1}{p^{8+15s}-1},
$$
respectively.
Finally, for $I=\emptyset$, we have the contribution of $E_\emptyset^\circ \cong \A^2 \setminus \{f=0\}$, being $p^2 - p$.

Adding everything, we expect to obtain $(p^{1+s}-1)(p^{2+3s}-1)(p^{3+5s}-1)(p^{5+9s}-1)(p^{8+15s}-1) $ as common denominator, but it is an exercise to check that a substantial simplification occurs: after cancellation, 	$Z(f;s)$ turns out to be of the form
$$ \frac{A(p^s)}{(p^{1+s}-1)(p^{8+15s}-1)}, $$
where $A(p^s)$ is a polynomial in $p^s$ (of degree 16).
We will see later that this cancellation phenomenon is strongly related to the monodromy conjecture.

According to Remark \ref{generalDenef}(2), one computes $Z_0(f;s)$ similarly.  Out of the above ten terms, we just omit the terms coming from $E_\emptyset^\circ$ and $E_0^\circ$.
\end{example}

\begin{remark}  The example above is pretty easy, in the sense that one can easily verify that $h$ has good reduction mod $p$ (see Remark \ref{generalDenef}(1)) for all $p$.  In general there will be primes $p$ for which $h$ does not have good reduction mod $p$, and it is quite difficult to determine these primes.
\end{remark}

\subsection{The topological zeta function}\label{subsec:topological zeta function}

We now indicate a funny heuristic argument to compute the limit of the expression (\ref{Denef's formula}), when $p$ tends to $1$ (whatever that means \dots).
First, using l'H\^opital's rule, $\frac{p-1}{p^{\nu_i+N_i s}-1}$ tends to $\frac{1}{\nu_i+N_i s}$. 	
More challenging, by Grothendieck's trace formula,  $\sharp (\overline{E_I}^\circ(\F_p))$  can be written as an alternating sum of traces of the $p$th power Frobenius operator, acting on adequate $\ell$-adic cohomology groups of $\overline{E_I}^\circ$ \cite{Denef2}. The limit when $p$ tends to $1$ of this operator is morally the identity operator, and then that alternating sum tends to the alternating sum of the dimensions of those cohomology groups, that is, to an $\ell$-adic Euler characteristic. Finally, a comparison theorem motivates to consider the usual topological Euler characteristic\index{Euler characteristic} $\chi(E_I^\circ)$ of the complex points of $E_I^\circ$ as the limit of $\sharp (\overline{E_I}^\circ(\F_p))$.
	We conclude that heuristically

	\[
	\lim_{p\rightarrow 1} Z(f;s)=\sum_{I\subset S}\chi(E_I^\circ)\prod_{i\in I}\frac{1}{\nu_i+N_is}.
	\]

\noindent
Denef and Loeser were inspired by this heuristic argument to define a new singularity invariant for any $f\in \C[x_1,\dots,x_n]\setminus \C$ \cite{DenefLoeser1}.

	\begin{definition}
		Let $f\in \C[x_1,\dots,x_n]\setminus \C$ and choose an embedded resolution $h:Y\rightarrow \A_\C^n$ of $\{f=0\}$, for which we use notation from Note \ref{notation embedded resolution}.
The (global) {\em topological zeta function of $f$\index{topological zeta function}} is
\begin{equation}
Z_{\top}(f;s) :=	\sum_{I\subset S}\chi(E_I^\circ)\prod_{i\in I}\frac{1}{\nu_i+ N_is}.
\end{equation}
The {\em  local topological zeta function\index{local topological zeta function}} of $f$ at $a\in \{f=0\}$ is
\begin{equation}
Z_{\top,a}(f;s) :=	\sum_{I\subset S}\chi(E_I^\circ \cap h^{-1}\{a\})\prod_{i\in I}\frac{1}{\nu_i+ N_is}.
\end{equation}
	\end{definition}
	
	\begin{remark}
It is not clear that the topological zeta function is well-defined, i.e., that it is independent of the chosen embedded resolution. Denef and Loeser proved this in  \cite{DenefLoeser1} by turning the heuristic discussion above into an exact argument, using $\ell$-adic interpolation.  An easier and more geometric approach is to use the {\em weak factorization theorem\index{weak factorization theorem}} \cite{AKMW}\cite{Wl1} (that was not known in 1991). It reduces the problem to showing that the defining expression of $Z_{\top}(f;s)$ is invariant under an admissible blow-up, which is an easy calculation.
	\end{remark}

\begin{example}\label{extop}(continuing Example \ref{example})
We compute $Z_{\top}(f;s)$ and $Z_{\top,0}(f;s)$  for $f=y^3-x^5$ using their defining formulas. Noting that $\chi(E_0^\circ) = \chi(E_3^\circ) = \chi(E_\emptyset^\circ) = 0$, both are equal to
\begin{align*}
&\frac 1{(2+3s)(5+9s)} + \frac 1{(8+15s)(5+9s)} + \frac 1{(8+15s)(1+s)} + \frac 1{(8+15s)(3+5s)} \\
&+ \frac 1{2+3s} + \frac 1{3+5s} + \frac{-1}{8+15s}.
\end{align*}
After simplification we obtain that
$$Z_{\top}(f;s) = Z_{\top,0}(f;s)  = \frac{8+7s}{(1+s)(8+15s)}.$$
This could also be achieved by \lq letting $p$ tend to $1$\rq\ in the final expressions for $Z(f;s)$ and  $Z_0(f;s)$ in Example \ref{example}.
\end{example}

\noindent	
The topological zeta function is an interesting singularity invariant on its own, and also a useful test case for studying (the poles of) the $p$-adic Igusa zeta function.
Despite its name, it is an analytic, but not a topological singularity invariant (except in the case of curves), see \cite{ACLM0}.
		
	\begin{problem}
It is a challenging open problem to give an intrinsic definition of the topological zeta function.
	\end{problem}

There is an obvious formulation for an analogous (global and local) monodromy conjecture for the topological zeta function  \cite{DenefLoeser1}.

\begin{conjecture}\label{mon conj topological}\index{monodromy conjecture}
Let $f\in \C[x_1,\dots,x_n]\setminus \C$.
\begin{enumerate}
\item[(1)]  If $s_0$ is a pole of $Z_{\top}(f;s)$, then $e^{2\pi i s_0}$ is a monodromy eigenvalue of $f:\C^n\to \C$ at some point of $\{f=0\}$.
\item[(2)] Assume that $f(0)=0$.  If $s_0$ is a pole of $Z_{\top,0}(f;s)$, then $e^{2\pi i s_0}$ is a monodromy eigenvalue of $f:\C^n\to \C$ at some point of $\{f=0\}$ close to $0$.
\end{enumerate}
\end{conjecture}
The stronger versions say that $s_0$ is then a root of $b_f$ or $b_{f,0}$, respectively.

\subsection{The motivic zeta function}\label{motivic}

Kontsevich suggested the idea of motivic integration\index{motivic integration} (a geometric analogue of $p$-adic integration) in a lecture at Orsay in 1995. This theory has been further developed by Denef and Loeser, amongst others; we refer to the wonderful book \cite{CNS} and the references therein.
The motivic zeta function is defined as an analogue of the Igusa zeta function, using motivic integration instead of $p$-adic integration.

	\begin{definition}
		Let $k$ be any field. The {\em Grothendieck group\index{Grothendieck group} of $k$-varieties} $K_0(\mathrm{Var}_k)$ is the quotient of the free abelian group generated by the symbols $[X]$, where $X$ runs over all varieties over $k$, by the relations
		\begin{align*}
			[X] = [Y] &\mbox{ if } X\cong Y,\\
			[X\setminus Y]+[Y]=[X] &\mbox{ if } Y \mbox{ is a closed subvariety of }X.
		\end{align*}
Multiplication determined by $[X]\cdot [Y]:=[X\times_k Y]$
		turns $K_0(\mathrm{Var}_k)$ into a commutative ring, called the {\em Grothendieck ring\index{Grothendieck ring} of $k$-varieties}.
We denote by $\mathbb{L}$ the class of the affine line $\A_k^1$ and by $K_0(\mathrm{Var}_k)_\mathbb{L}$ the localization of $K_0(\mathrm{Var}_k)$ with respect to $\mathbb{L}$.
	\end{definition}

The motivic zeta function \cite{DenefLoeser2}  is defined as an integral over the \lq arc space\index{arc space}\rq\ $k \llbracket t \rrbracket^n$, which is the analogue of the space of $p$-adic integers $\Z_p^n$. It is a power series in the formal variable $T$, which is sometimes written as $\mathbb{L}^{-s}$ (with a formal $s$), to stress the analogy with $p^{-s}$.
That arc space carries a natural {\em motivic measure\index{motivic measure} $d\mu$}, with values in $K_0(\mathrm{Var}_k)_\mathbb{L}$, which is the analogue of the Haar measure on $\Z_p^n$. For a detailed explanation, we refer e.g. to the original papers \cite{DenefLoeser2} and \cite{DenefLoeser3}, to the introductory survey \cite{Veys Sapporo} or to  \cite{CNS}.

	\begin{definition}
Let $k$ be a field of characteristic zero and $f\in k[x_1,\dots,x_n]\setminus k$. The  {\em motivic zeta function\index{motivic zeta function} of $f$} is a formal power series over $K_0(\mathrm{Var}_k)_\mathbb{L}$ in the variable $T=\mathbb{L}^{-s}$, given by the (converging) motivic integral
		\[Z_{\mot}(f;s)= Z_{\mot}(f;T):=\int_{k\llbracket t \rrbracket^n}(\mathbb{L}^{-s})^{\mathrm{ord}_tf(x)}d\mu.\]
	\end{definition}

\noindent	
More concretely, define $$\mathcal{X}_i=\left\{\gamma\in \left(\frac{k\llbracket t \rrbracket}{(t^{i+1})}\right)^n\mid \mathrm{ord}_t f_i(\gamma)=i\right\}$$
for $i\in \Z_{\geq 0}$, where $f_i:\left(\frac{k\llbracket t\rrbracket}{(t^{i+1})}\right)^n\rightarrow \frac{k\llbracket t\rrbracket}{(t^{i+1})}$ is the natural extension of $f:k^n \rightarrow k$, and for $\lambda \in \frac{k\llbracket t\rrbracket}{(t^{i+1})}$ we denote by $\mathrm{ord}_t(\lambda)\in \{0,1,\dots,i,+\infty\}$ the highest power of $t$ that divides $\lambda$. Then we have
	\[
	Z_{\mot}(f;s)=\frac{1}{\mathbb{L}^n}\sum_{i\geq 0}[\mathcal{X}_i]\mathbb{L}^{-in-is}.
	\]
In order to see also the analogy of this expression with the Igusa zeta function, one easily verifies that, for $f\in \Q_p[x_1,\dots,x_n]\setminus \Q_p$, we have
\[
	Z(f;s)=\frac 1{p^n}\sum_{i\geq 0} \sharp\left\{ a\in \left(\frac{\Z_p}{(p^{i+1})}\right)^n \mid \mathrm{ord}_p f(a) = i  \right\} p^{-in-is}.
	\]

Denef's formula\index{Denef's formula} can be upgraded to a similar formula for $Z_{\mot}(f;s) $ in terms of an embedded resolution of $\{f=0\}$  \cite{DenefLoeser2}.
The various versions of the monodromy conjecture\index{monodromy conjecture} can also be formulated for the motivic zeta function. In this setting there is a subtlety in the notion of pole because $K_0(\mathrm{Var}_k)$ is not a domain \cite{Po} (even $\mathbb{L}$ is a zero divisor \cite{Borisov}), see \cite{RV2}.
For precise statements we refer to  \cite{Nicaise}\cite{DenefLoeser2}.  %, and for an upgrade of  proof for $n=2$ in  [Corollary 8.2.2]{BN}

\medskip
The motivic zeta function specializes to the $p$-adic one for all but finitely many $p$ and to the topological one \cite{DenefLoeser2}.  Hence the monodromy conjecture in the motivic setting implies it in the two other settings, but not necessarily in the other direction.
(There exist explicit examples, like $f=x_1^3+x_2^3+x_3^3+x_4^3+x_5^6$, where a pole of the motivic zeta function vanishes for the topological zeta function.)

However,  crucial arguments in proofs are often of geometric nature and typically \lq universal\rq. In that case proofs for the topological zeta function can immediately be upgraded to proofs for the motivic zeta function. In particular, this holds for the arguments in \S\ref{sec:3} and \S\ref{sec:surfaces}.  Hence we focus in the sequel, for sake of exposition,  on the topological zeta function: formulas are shorter and that way we can concentrate on the crux of the arguments.

\subsection{Monodromy and b-function}

We now also describe the \lq other side\rq\ of the monodromy conjecture, monodromy eigenvalues and roots of the b-function, in terms of an embedded resolution.

\begin{definition}
Let $f\in \C[x_1,\dots,x_n]\setminus \C$ and fix a point $a\in \{f=0\}$. Let $F_a$ denote the Milnor fibre of $f$ at $a$, and $P_i(t)$ the characteristic polynomial of the monodromy $M_i$ acting on $H^i(F_a,\C)$ for $i=0,\dots,n-1$. The {\em monodromy zeta function\index{monodromy zeta function}} of $f$ at $a$ is the alternating product
$$
\zeta_a(t):= \prod_{i=0}^{n-1} P_i(t)^{(-1)^{i+1}} = \frac {P_1(t)\cdot P_3(t)\cdot \dots}{P_0(t)\cdot P_2(t)\cdot \dots}.
$$
\end{definition}
(We note that some authors use a different convention.)

\begin{theorem}[A'Campo's formula \cite{A'C}]\label{A'Campo}\index{A'Campo's formula}
Let $f\in \C[x_1,\dots,x_n]\setminus \C$ and choose an embedded resolution $h:Y\rightarrow \A_\C^n$ of $\{f=0\}$, for which we use notation from Note \ref{notation embedded resolution}.

(1)  For any point $a\in \{f=0\}$, we have that
$$
\zeta_a(t)= \prod_{j\in S}(t^{N_j}-1)^{-\chi\big(E_j^{\circ} \cap h^{-1}\{a\}\big)}.
$$

(2) When $a$ is an isolated singularity of $\{f=0\}$, the formula simplifies considerably, yielding an expression for $P_{n-1}(t)$.  More precisely,
\begin{eqnarray*}
P_{n-1}(t) &=& (t-1)\zeta_a(t) = (t-1) \prod_{j\in S_e}(t^{N_j}-1)^{-\chi(E_j^{\circ})}  \qquad (n\text{ even}),  \\
P_{n-1}(t) &=& \frac{1}{(t-1)\zeta_a(t)} = \frac1{t-1} \prod_{j\in S_e}(t^{N_j}-1)^{\chi(E_j^{\circ})}  \qquad (n\text{ odd}).
\end{eqnarray*}
\end{theorem}

For non-isolated singularities, $\zeta_a(t)$ can miss some eigenvalues, due to cancellation in numerator and denominator. However, these are then detected by monodromy zeta functions at nearby points \cite[Lemma 4.6]{Denef3}.
\begin{proposition}
Let $f\in \C[x_1,\dots,x_n]\setminus \C$ and fix a point $a\in \{f=0\}$.  The set of all monodromy eigenvalues of $f$ at all points $b$ in a small enough neighbourhood of $a$ is equal to the set of all zeroes and poles of  $\zeta_b(t)$ for all points $b$ in a small enough neighbourhood of $a$.
\end{proposition}

\begin{example}(continuing Example \ref{example})
For $f=y^3-x^5$ we have by Theorem \ref {A'Campo} that
$$\zeta_0(t) = \frac{t^{15}-1}{(t^3-1)(t^5-1)}, \quad P_1(t)=  \frac{(t-1)(t^{15}-1)}{(t^3-1)(t^5-1)} \quad \text{ and }\quad P_0(t) = t-1.  $$
Hence the monodromy eigenvalues of $f$ at $0$ are $1$ and all primitive $15$th roots of unity.
\end{example}

\begin{remark}\label{expectation}
Looking at Theorem \ref{A'Campo},  we expect \lq generically\rq\ that, if $\chi(E_j^{\circ})=0$, then $-\frac{\nu_j}{N_j}$ is {\em not} a pole of $Z_{\top}(f;s)$.
More generally, we can expect this as soon as $\chi(E_j^{\circ}) \geq 0$ for even $n$, and as soon as $\chi(E_j^{\circ}) \leq 0$ for odd $n$. In the next sections, this idea will be an important part of a proof for curves and a strategy for surfaces.
\end{remark}

For roots of the b-function, there is no exact formula (except in some special cases), but we do have a complete list of possible candidate roots.

\begin{theorem}[\cite{Li} \cite{Ko}]\label{roots b}\index{b-function}\index{roots of b-function}
Let $K$ be a field of characteristic $0$ and $f\in K[x_1,\dots,x_n] \setminus  K$. Let $h:Y\rightarrow \A_K^n$ be an embedded resolution of $\{f=0\}$, for which we use notation from Note \ref{notation embedded resolution}.  Then the roots of $b_f$ are of the form
$ -\frac{\nu_j+k}{N_j}$  with $ j\in S$  and $k\in \Z_{\geq 0}$.
\end{theorem}

Of course, only finitely many \lq shifts\rq\ by $k$ are possible. Moreover, all roots of $b_f$ are contained in the interval $(-n,0)$, see e.g. \cite{Sa}.

\begin{example}(continuing Example\ref{example})
For $f=y^3-x^5$ we have that
$$b_f = \Big(s+\frac{22}{15}\Big)\Big(s+\frac{19}{15}\Big)\Big(s+\frac{17}{15}\Big)\Big(s+\frac{16}{15}\Big)\Big(s+1\Big)\Big(s+\frac{14}{15}\Big)\Big(s+\frac{13}{15}\Big)\Big(s+\frac{11}{15}\Big)\Big(s+\frac{8}{15}\Big).$$
This follows from a formula for isolated quasihomogeneous singularities; see e.g. \cite[\S11]{Yano}.
\end{example}

\begin{remark} Both in Example \ref{first example} and in the example above there is a bijection between the roots of the b-function and the monodromy eigenvalues; this is not true in general. Take for example $f = (y^2-x^3)^5-x^{18}$, then $-7/30$ and  $-37/30$ are roots of $b_f$, both inducing the eigenvalue $e^{2\pi i (-7/30)}$.
\end{remark}

The largest root of $b_f$ is related to an important singularity invariant, namely the log canonical threshold\index{log canonical threshold} of $f$.

\begin{definition}
Let $f\in \C[x_1,\dots,x_n]\setminus \C$ and $a\in \{f=0\}$. Let $h:Y\rightarrow \A_\C^n$ be an embedded resolution of singularities of $\{f=0\}$, for which we use notation from Note \ref{notation embedded resolution}.
The (global) {\em log canonical threshold of $f$} and the (local) {\em log canonical threshold of $f$ at $a$} are
\begin{equation}\label{lct}
\lct(f) := \min_{j\in S} \frac{\nu_j}{N_j} \qquad\text{and}\qquad \lct_a(f) := \min_{j\in S, a\in f(E_j)} \frac{\nu_j}{N_j},
\end{equation}
respectively.
\end{definition}

One can show that these invariants are independent of the chosen embedded resolution. Alternatively, there are several intrinsic definitions for the log canonical threshold, and (\ref{lct}) can be considered as formulas derived from one of these definitions. We refer to e.g. the survey paper  \cite{Mu1} for details. Yet another argument for independence is the announced relation with the b-function.

\begin{theorem}[\cite{Ka}]
Let $f\in \C[x_1,\dots,x_n]\setminus \C$.

(1)  Then $\lct(f)$ is the absolute value of the largest root of $b_f$.

(2) Take $a\in \{f=0\}$. Then $\lct_a(f)$ is the absolute value of the largest root of $b_{f,a}$.
\end{theorem}

\section{The case of plane curves ($n=2$)}
\label{sec:3}

Let $f\in \C[x,y]\setminus \C$.  Working globally, we take an embedded resolution of singularities $h:Y\rightarrow \A^2_\C$ of $\{f=0\}$. Working locally, we assume that $f(0)=0$, and we take an embedded resolution $h$ of the germ of $\{f=0\}$ at $0$, that is, of the restriction of $\{f=0\}$ to a small enough neighbourhood of $0$.  We use notation from Note \ref{notation embedded resolution}.

In this dimension the formulas for the topological zeta function are more concretely
\begin{equation}\label{top zeta function dim 2}
\begin{aligned}
&Z_{\top}(f;s) = \chi(Y\setminus \cup_{j\in S}E_j) + \sum_{j\in S} \frac{\chi(E_j^\circ)}{\nu_j+N_js} + \sum_{i\neq j \in S}   \frac{\chi(E_i \cap E_j)}{(\nu_i+N_is)(\nu_j+N_js)},
\\
&Z_{\top,0}(f;s) = \sum_{j\in S_e} \frac{\chi(E_j^\circ)}{\nu_j+N_js} + \sum_{i\neq j \in S}   \frac{\chi(E_i \cap E_j)}{(\nu_i+N_is)(\nu_j+N_js)},
\end{aligned}
\end{equation}
where we recall that $\chi(Y\setminus \cup_{j\in S}E_j)=\chi(\C^2\setminus \{f=0\})$.
Note that $s_0$ is a pole of order 2 of $Z_{\top}(f;s)$ or $Z_{\top,0}(f;s)$ if and only if there exist two intersecting components $E_i$ and $E_j$ such that $s_0=- \frac{\nu_i}{N_i}=-\frac{\nu_j}{N_j}$.
(A candidate pole of order 2 is always a pole of order 2 since $\chi(E_i \cap E_j)>0$ whenever $E_i \cap E_j \neq \emptyset$.)

\smallskip
We fix an exceptional component $E_0$, intersecting exactly $r$ times other components, say $E_1,\dots,E_r$.
Since $E_0$ is isomorphic to the complex projective line, we have that  $\chi(E_0)=2$ and $\chi(E_0^{\circ}) =2-r$.
In the formulas (\ref{top zeta function dim 2}), the contribution of $E_0$ is
\begin{equation}\label{contribution E}
\frac1{\nu_0 +N_0s}\Big(2-r + \sum_{i=1}^r \frac1{\nu_i +N_is}\Big).
\end{equation}
Say $\frac{\nu_0}{N_0} \neq \frac{\nu_i}{N_i}$ for $i=1,\dots,r$. Then $-\frac{\nu_0}{N_0}$ is a candidate pole of order $1$ of (\ref{contribution E}), with residue equal to
\begin{equation}\label{residue}
\frac1{N_0}\Big(2-r + \sum_{i=1}^r \frac1{\nu_i -(\nu_0/N_0)N_i}\Big).
\end{equation} We say that $E_0$ {\em does not contribute to the poles\index{non-contribution to poles}} of $Z_{\top}(f;s)$ or $Z_{\top,0}(f;s)$ if this residue is zero.

\subsection{Non-contribution}

In the context of the expectation in Remark \ref{expectation}, we verify that $E_0$ does not contribute to the poles of  $Z_{\top}(f;s)$ or $Z_{\top,0}(f;s)$ when $\chi(E_0^\circ)\geq 0$, that is, when $r=1$ or $r=2$.
This is an immediate consequence of the following lemma, that was first proven for arbitrary plane curves by Loeser \cite{Loeser} (preceded by some partial results by Strauss, Meuser and Igusa). We present a short conceptual proof, which is also the starting point of a generalized theory of relations and congruences between numerical data in arbitrary dimension $n$ \cite{Veys Relations}\cite{Veys Congruences}\cite{Veys More}.

\begin{lemma} \label{relationsbetweennumericaldata}\index{relations between numerical data}
Let the exceptional component $E_0$  intersect exactly $r$ times other components, say $E_1,\dots,E_r$.
 Denote $\kappa=-E_0^2$, where $E_0^2=E_0\cdot E_0$ is the self-intersection number of $E_0$ on $Y$.  Then
\begin{itemize}
\item[(1)]  \   $\kappa N_0=\sum_{i=1}^rN_i$, and
\item[(2)] \  $\kappa \nu_0=\sum_{i=1}^{r}(\nu_i-1)+2$.
\item[(3)] \   Denote also $\alpha_i:= \nu_i - \frac{\nu_0}{N_0} N_i$ for $i=1,\dots,r$. Then $\sum_{i=1}^{r}(\alpha_i-1)+2=0$.
\end{itemize}
\end{lemma}
\begin{proof}
In the Picard group $\Pic(Y)$ of $Y$, that is, the group of Cartier divisors
 on $Y$ modulo linear equivalence, we have that
$$
\sum_{j\in S} N_jE_j =h^*\mathrm{div}(f)=\mathrm{div}(h^*(f))=0 \quad \text{and}\quad
\sum_{i\in S} (\nu_j-1)E_j =K_Y-h^*K_{\A^2}=K_Y .
$$
We compute the intersection product of both expressions with $E_0$ (in $\Pic(Y)$). First,
$$
0=E_0\cdot 0 = E_0 \cdot \big( N_0E_0 + \sum_{j \neq 0} N_jE_j \big)
			=N_0E_0^2 +  \sum_{i=1}^rN_i,
$$
yielding (1). Similarly, we compute
$$
E_0\cdot K_Y=	E_0\cdot\big( (\nu_0-1)E_0+\sum_{j\neq 0}(\nu_j-1)E_j\big)
= 	(\nu_0-1)E_0^2+ \sum_{i=1}^r(\nu_i-1).
$$
Using the adjunction formula, we obtain
$$
-2=\mathrm{deg}\,  K_{E_0}=E_0 \cdot (K_Y+E_0)= \nu_0 E_0^2+\sum_{i=1}^r(\nu_i-1),
$$
yielding (2). Finally, (3) follows immediately from (1) and (2).
\end{proof}
We note that (3) could also be proven directly, as a different application of the adjunction formula, see e.g. \cite{Veys Relations}.

\begin{proposition}\label{non contribution}\index{non-contribution to poles}
Using notation above, $E_0$ does not contribute to the poles of $Z_{\top}(f;s)$ or $Z_{\top,0}(f;s)$ when $r=1$ or $r=2$.
\end{proposition}
\begin{proof}
The residue in these cases is $\frac1{N_0}(1+\frac1{\alpha_1})$ and $\frac1{N_0}(\frac1{\alpha_1}+\frac1{\alpha_2})$, respectively, and these expressions are zero by Lemma \ref{relationsbetweennumericaldata}(3).
\end{proof}

\subsection{Proof of the conjecture}

We present a simplification of the proofs in the literature.
In his proof in \cite{Loeser}, Loeser uses two crucial properties of the numerical data: Lemma \ref{relationsbetweennumericaldata}(3) and the following inequalities \cite[Proposition II.3.1]{Loeser}.

\begin{lemma} \label{inequalitiesbetweennumericaldata}\index{inequalities between numerical data}
Assume that $h$ is the {\em minimal} embedded resolution of singularities of $\{f=0\}$ or of its germ at $0$, respectively.
Let the exceptional component $E_0$  intersect exactly $r$ times other components, say $E_1,\dots,E_r$.
Denote  $\alpha_i:= \nu_i - \frac{\nu_0}{N_0} N_i$ for $i=1,\dots,r$. Then $-1 \leq \alpha_i < 1$, with $\alpha_i  =-1$ if and only if $r=1$.
\end{lemma}

\begin{theorem}  (1)  Let $f\in \C[x,y]\setminus \C$, satisfying $f(0)=0$. If $s_0$ is a pole of  $Z_{\top,0}(f;s)$,   then $e^{2\pi i s_0}$ is a monodromy eigenvalue of $f$ at some point of $\{f=0\}$, close to $0$.

(2) Let $f\in \C[x,y]\setminus \C$. If $s_0$ is a pole of  $Z_{\top}(f;s)$,   then $e^{2\pi i s_0}$ is a monodromy eigenvalue of $f$ at some point of $\{f=0\}$.
\end{theorem}

\begin{proof}\index{proof of monodromy conjecture for curves} Write $s_0=-\frac cd$, where $c$ and $d$ are coprime positive integers.

\smallskip
(1) % Let $h$ be the minimal embedded resolution of the germ of $\{f=0\}$ at $0$.
By Proposition \ref{monodromy properties}(3), we have for each component $E_j, j\in S_s,$ that all $N_j$-th roots of unity are monodromy eigenvalues of $f$ at points of $E_j$ close to $0$. In particular, if $d$ divides such $N_j$, then $e^{2\pi i s_0}$ is a monodromy eigenvalue of $f$ at such a point.
  We may thus assume that $s_0 = -\frac{\nu_j}{N_j}$ for some $j\in S_e$, and moreover that $d \not\mid N_j$ for $j\in S_s$.
%By Proposition \ref{monodromy properties}(3), if $s_0=-\frac{1}{N_j}$ for $j\in S_s$, then $e^{2\pi i s_0}$ is an eigenvalue of $f$ at points of the corresponding component of $\{f=0\}$ close to $0$. We may thus assume that $s_0\in \{-\frac{\nu_j}{N_j}\mid j\in S_e\} \setminus  \{-\frac{1}{N_j}\mid j\in S_s\}$.
%Write $s_0=-\frac cd$, where $c$ and $d$ are coprime positive integers.
 From A'Campo's formula (Theorem \ref{A'Campo}), $e^{2\pi i s_0}$ is a monodromy eigenvalue of $f$ at 0 if $\sum_{j\in S_e, d|N_j} \chi(E_j^\circ) <0$.

We denote by $\mathcal{C}_\ell, \ell \in L,$ the connected components of  $\cup_{j\in S_e, d|N_j}E_j$.  We claim that
$$
\sum_{E_j\subset \mathcal{C}_\ell} \chi(E_j^\circ) \leq 0  \qquad\text{ for each } \ell\in L.
$$
Indeed, fix $\ell\in L$ and  let $P_1,\dots,P_t$ be the intersection points of $\mathcal{C}_\ell$ with components $E_k$ not belonging to $\mathcal{C}_\ell$.  Note that $d \not\mid N_j$ for such an intersecting component $E_k$ and that necessarily $t\geq 1$.
We use now that $\mathcal{C}_\ell$ is a tree of $\P^1$'s. Its \lq open part\rq, that is, $\mathcal{C}_\ell$ without all intersections points of its components, is precisely $(\cup_{E_j\subset \mathcal{C}_\ell} E_j^\circ) \cup \{P_1, \dots, P_t\}$. It is easy to see that the Euler characteristic of the open part of a tree of $\P^1$'s is $2$, and hence $\sum_{E_j\subset \mathcal{C}_\ell} \chi(E_j^\circ) = 2 -t$.  Lemma \ref{relationsbetweennumericaldata}(1) implies that $t \geq 2$, finishing the proof of the claim.  We will show next that $\sum_{E_j\subset \mathcal{C}_\ell}  \chi(E_j^\circ)   <0$ for some $\ell$, which then establishes the proof of the theorem.

As a preparation, take any exceptional component $E_0$ with $s_0=-\frac{\nu_0}{N_0}=-\frac cd$. Say it intersects exactly $r$ times other components $E_1,\dots,E_r$, for which we denote again $\alpha_i:= \nu_i - \frac{\nu_0}{N_0} N_i$ for $i=1,\dots,r$. Either $r=1$, and then $\alpha_1 = -1$ and  $\frac{\nu_0}{N_0} \neq \frac{\nu_1}{N_1}$, or $r\geq 2$, and then
\begin{equation}\label{alpha0}
d|N_i \Leftrightarrow \alpha_i \in \Z \Leftrightarrow  \alpha_i =0 \Leftrightarrow \frac{\nu_0}{N_0} =\frac{\nu_i}{N_i},
\end{equation}
for $i=1,\dots,r$, where the second equivalence follows from Lemma \ref{inequalitiesbetweennumericaldata}.

\smallskip
\noindent {\em First case: $s_0$ is a pole of order $1$.} By Proposition \ref{non contribution}, there is at least one exceptional curve, say $E_0$, with $s_0=-\frac{\nu_0}{N_0}$, intersecting $r\geq 3$ times other components.  By (\ref{alpha0}) all $E_i$ intersecting $E_0$ satisfy $d\not\mid N_i$, and hence $E_0$ is itself a connected component $\mathcal{C}_\ell$ for which $\chi(E_0^\circ)\ <0$.

\smallskip
\noindent {\em Second case: $s_0$ is a pole of order $2$.} Hence there exist intersecting $E_0$ and $E'_0$ satisfying $s_0=-\frac{\nu_0}{N_0}=-\frac{\nu'_0}{N'_0}$.  Using Lemma \ref{relationsbetweennumericaldata}(1)  and (\ref{alpha0}), it is not difficult to conclude that the connected component $\mathcal{C}_\ell$ containing $E_0$ and $E'_0$  is a \lq chain\rq, as presented in the figure below, consisting of say $E_1,  E_2, \dots, E_m \, (m\geq 2)$, where, for $2\leq i \leq m-1$, $E_i$ intersects exactly $E_{i-1}$ and $E_{i+1}$, and $E_1$ and $E_m$ each intersect at least two  other components $E_k$, all satisfying $d \not\mid N_k$. Hence $\sum_{E_j\subset \mathcal{C}_\ell} \chi(E_j^\circ) =\sum_{j=1}^m \chi(E_j^\circ) \leq -2$
	
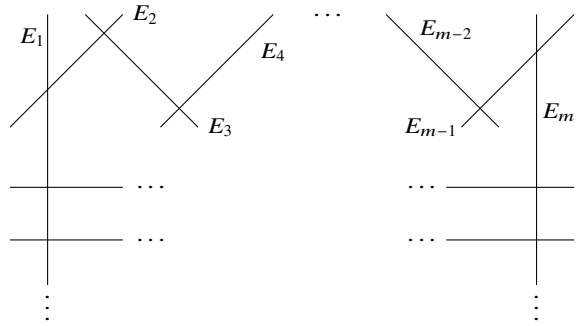
\begin{figure}[H]
			\centering
			\begin{tikzpicture}
			\useasboundingbox (-0.5,0) rectangle (14,4.9);
			\draw (2,0.9) -- (2,4.5);
\node at (2,0.7) {$\vdots$};	
			\draw (1.5,3) -- (3,4.5);
			\draw (2.5,4.5) -- (4,3);
			\draw (3.5,3) -- (5,4.5);
			\draw (6.5,4.5) -- (8,3);
						\draw (7.5,3) -- (9,4.5);
			\draw (1.5,1.5) -- (3,1.5);
	\draw (1.5,2.2) -- (3,2.2);
\draw (8.5,0.9) -- (8.5,4.5);
\node at (8.5,0.7) {$\vdots$};
\node at (3.4,1.5) {$\dots$};	
\node at (3.4,2.2) {$\dots$};	
	\draw (7.3,2.2) -- (9,2.2);
	\draw (7.3,1.5) -- (9,1.5);
\node at (7,1.5) {$\dots$};	
\node at (7,2.2) {$\dots$};	
						\node at (5.75,4.5) {$\dots$};							
			\node at (1.8,4.2) {$E_{1}$};
						\node at (3.3,4.5) {$E_{2}$};
			\node at (4.3,3) {$E_{3}$};
			\node at (5,4) {$E_{4}$};
			\node at (7.3,4.3) {$E_{m-2}$};
			\node at (7.1,3) {$E_{m-1}$};
\node at (8.8,3.2) {$E_{m}$};	
						\end{tikzpicture}
		\end{figure}
%\smallskip
(2)   As above, we may assume that $s_0 = -\frac{\nu_j}{N_j}$ for some $j\in S_e$, and moreover that $d \not\mid N_j$ for $j\in S_s$.
%$s_0\in \{-\frac{\nu_j}{N_j}\mid j\in S_e\} \setminus  \{-\frac{1}{N_j}\mid j\in S_s\}$,
Then $s_0$ must be a pole of a local zeta function $Z_{\top,a}(f;s)$ for some $a\in \{f=0\}$ and hence, as established in the proof of (1), $e^{2\pi i s_0}$ is a monodromy eigenvalue of $f$ at $a$.
\end{proof}

\begin{remark}
Loeser's proof of the strong version, involving roots of the b-function \cite{Loeser}, also starts with showing that a pole of the zeta function can only arise from an exceptional component $E_j$ intersecting at least three times other components (or from the strict transform). In order to show that such a pole is a root of the b-function, more heavy machinery is needed: asymptotic expansions of certain fibre integrals, and a result of Deligne and Mostow \cite{DM}, constructing a nonzero cohomology class of a certain local system on $E_j^\circ$.
\end{remark}

\begin{remark} The strong version of the monodromy conjecture even has a more precise formulation, saying that $b_{f,0}(s) \cdot Z_{\top,0}(f;s)$ is a polynomial, that is, the order of a pole of $Z_{\top,0}(f;s)$ is at most its multiplicity as root of $b_{f,0}(s)$.  When $n=2$, this just means as extra statement that a pole of order $2$ is a root of multiplicity $2$. In  \cite{Loeser}, Loeser verified it when $f$ is reduced.  (In fact, $Z_{\top,0}(f;s)$ has at most one pole of order $2$, see below, which was not known at that time.)

For more results concerning order of poles versus multiplicity of roots, and also sizes of Jordan blocks of monodromy, see \cite{MTV1}\cite{MTV2}.%\cite{Bl}.
\end{remark}

\begin{remark}
For $n=2$, a variant of the topological zeta function and the monodromy conjecture for meromorphic germs was treated in \cite{GL}.
\end{remark}

\subsection{Structure of resolution graph and determination of all poles}

Considering  the minimal embedded resolution of $f$, a finer use of Lemma \ref{relationsbetweennumericaldata} and Lemma \ref{inequalitiesbetweennumericaldata} yields
\begin{itemize}
\item  a nice structure theorem for its resolution graph, which is of independent interest,
\item  a conceptual description of {\em all} poles of $Z_{\top,0}(f;s)$,
\item  the fact that  $Z_{\top,0}(f;s)$ has at most one pole of order two, which is then minus the log canonical threshold of $f$ at $0$.
\end{itemize}
Below we mention the precise results, and refer to the original paper for proofs. These are local results; so  we consider only the germ of $f$ at the origin. (We also exclude the trivial case where this germ has already normal crossings.)

Let $h$ be the {\em minima}l embedded resolution of the germ  of $\{f=0\}$ at $0$. The {\em dual graph}\index{dual resolution graph} associated to this resolution consists of the following data. The vertices of  the graph correspond to the irreducible components $E_j, j\in S$, where the exceptional components are represented by a dot and the (analytically irreducible) components of the strict transform by a circle. Two vertices are connected by an edge precisely when the associated components intersect. It is well known that this dual graph is a tree, where all circles are end vertices.
We decorate the vertex corresponding to $E_j$ with the value $\frac{\nu_j}{N_j}$.

A vertex with at least three edges is depicted in the following way.
\begin{figure}[H]
	\centering
	\begin{tikzpicture}[scale=0.8]
	\vertex (a)[white] at (0,1) [label=below:]{};
	\vertex (b)[white] at (0,0.33) [label=below:]{};
	\vertex (c)[white] at (0,-0.33) [label=below:]{};
	\vertex (d)[white] at (0,-1) [label=above:]{};
	\vertex (e)[fill] at (1,0) [label=above:]{};
	\vertex (f)[white] at (3,0) [label=above:]{};
	
	\path
	(b) edge (e)[dashed]
	(c) edge (e)
	;
	\path
	(a) edge (e)
	(d) edge (e)
	(e) edge (f)
	;
	\end{tikzpicture}
\end{figure}

\begin{theorem}[\cite{Veys Determination}]
\label{stuctureofdualgraph2}\index{dual resolution graph}
Let $f\in \C[x,y]\setminus \C$, satisfying $f(0)=0$. Consider the {\em minimal} embedded resolution $h$ of the germ of $\{f=0\}$ at $0$, where we use notation from Note
 \ref{notation embedded resolution}. Also, denote by $\mathcal{M}$ the locus (with edges) where $\frac{\nu_j}{N_j}$ is minimal. Then $\mathcal{M}$ is connected and has one of the following possible forms, where $r\geq 0$.
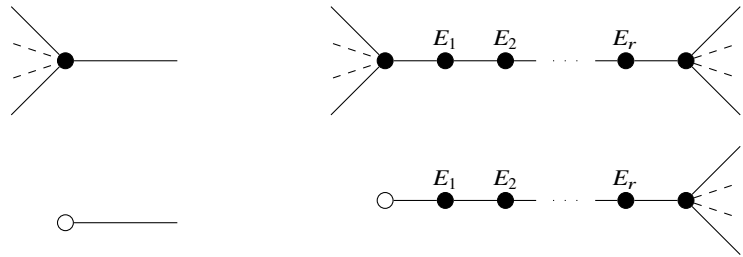
\begin{figure}[H]
	\centering
	\begin{minipage}{.4\textwidth}
		\centering
	\begin{tikzpicture}[scale=0.8]
	\vertex (a)[white] at (0,1) [label=below:]{};
	\vertex (b)[white] at (0,0.33) [label=below:]{};
	\vertex (c)[white] at (0,-0.33) [label=below:]{};
	\vertex (d)[white] at (0,-1) [label=above:]{};
	\vertex (e)[fill] at (1,0) [label=above:]{};
	\vertex (f)[white] at (3,0) [label=above:]{};
	
	\path
	(b) edge (e)[dashed]
	(c) edge (e)
	;
	\path
	(a) edge (e)
	(d) edge (e)
	(e) edge (f)
	;
			
	\end{tikzpicture}
	\end{minipage}%
	\begin{minipage}{.6\textwidth}
		\centering
	\begin{tikzpicture}[scale=0.8]
\vertex (a)[white] at (0,1) [label=below:]{};
\vertex (b)[white] at (0,0.33) [label=below:]{};
\vertex (c)[white] at (0,-0.33) [label=below:]{};
\vertex (d)[white] at (0,-1) [label=above:]{};
\vertex (e)[fill] at (1,0) [label=above:]{};
\vertex (f)[fill] at (2,0) [label=above:]{};
\vertex (g)[fill] at (3,0) [label=above:]{};
\vertex (h1)[white] at (3.65,0) [label=above:]{};
\vertex (h2)[white] at (4.35,0) [label=above:]{};
\vertex (i)[fill] at (5,0) [label=above:]{};
\vertex (j)[fill] at (6,0) [label=above:]{};
\vertex (k)[white] at (7,1) [label=below:]{};
\vertex (l)[white] at (7,0.33) [label=below:]{};
\vertex (m)[white] at (7,-0.33) [label=below:]{};
\vertex (n)[white] at (7,-1) [label=above:]{};

\node at (2,0.35) {$E_{1}$};
\node at (3,0.35) {$E_{2}$};
\node at (5,0.35) {$E_{r}$};

\path
(b) edge (e)[dashed]
(c) edge (e)
(j) edge (l)
(j) edge (m)
;
\path
(h1) edge (h2)[loosely dotted]
;
\path
(a) edge (e)
(d) edge (e)
(e) edge (f)
(f) edge (h1)
(h2) edge (i)
(i) edge (j)
(j) edge (k)
(j) edge (n)
;
\end{tikzpicture}
	\end{minipage}
	\centering
	\begin{minipage}{.4\textwidth}
		\centering
			\begin{tikzpicture}[scale=0.8]
		\vertex (E2)[white] at (0,1) [label=below:]{};
		\vertex (E1)[white] at (0,0) [label=below:]{};
		\vertex (e) at (1,0) [label=above:]{};
		\vertex (f)[white] at (3,0) [label=above:]{};

		\path
		(e) edge (f)
		;
				\end{tikzpicture}
	\end{minipage}%
	\begin{minipage}{.6\textwidth}
		\centering
		\begin{tikzpicture}[scale=0.8]
	\vertex (a)[white] at (0,1) [label=below:]{};
	\vertex (b)[white] at (0,0.33) [label=below:]{};
	\vertex (c)[white] at (0,-0.33) [label=below:]{};
	\vertex (d)[white] at (0,-1) [label=above:]{};
	\vertex (e) at (1,0) [label=above:]{};
	\vertex (f)[fill] at (2,0) [label=above:]{};
	\vertex (g)[fill] at (3,0) [label=above:]{};
	\vertex (h1)[white] at (3.65,0) [label=above:]{};
	\vertex (h2)[white] at (4.35,0) [label=above:]{};
	\vertex (i)[fill] at (5,0) [label=above:]{};
	\vertex (j)[fill] at (6,0) [label=above:]{};
	\vertex (k)[white] at (7,1) [label=below:]{};
	\vertex (l)[white] at (7,0.33) [label=below:]{};
	\vertex (m)[white] at (7,-0.33) [label=below:]{};
	\vertex (n)[white] at (7,-1) [label=above:]{};
	
	\node at (2,0.35) {$E_{1}$};
	\node at (3,0.35) {$E_{2}$};
	\node at (5,0.35) {$E_{r}$};
	
	\path
	(j) edge (l)[dashed]
	(j) edge (m)
	;
	\path
	(h1) edge (h2)[loosely dotted]
	;
	\path
	(e) edge (f)
	(f) edge (h1)
	(h2) edge (i)
	(i) edge (j)
	(j) edge (k)
	(j) edge (n)
	;
	\end{tikzpicture}
	\end{minipage}
\end{figure}
\noindent
(When $f$ is reduced, the last two cases, involving a component of the strict transform, cannot occur.)
	Furthermore, starting from $\mathcal{M}$, the values $\frac{\nu_i}{N_i}$ strictly increase along each path away from $\mathcal{M}$.
	\end{theorem}

Theorem \ref{stuctureofdualgraph2} implies in fact  that  $Z_{\top,0}(f;s)$ has {\em at most one} pole of order $2$, which is then minus
 the  log canonical threshold of $f$ at $0$.

\begin{remark}\index{log canonical threshold}
(1) The following generalization to higher dimensions was conjectured  in \cite{LaeremansVeys}, and proven by Nicaise and Xu \cite{NicaiseXu}.
	Let $f\in \C[x_1,\dots,x_n]\setminus \C$.  A pole of order $n$ of $Z_{\top,0}(f;s)$ must be equal to minus the  log canonical threshold of $f$ at $0$. In particular, there is at most one  pole of order $n$.

(2) Let again $f\in \C[x_1,\dots,x_n]\setminus \C$. A pole of order $n$ of $Z_{\top,0}(f;s)$ is always of the form $-\frac1k$ for some positive integer $k$ \cite{LaeremansVeys}.
\end{remark}	

\begin{theorem}[\cite{Veys Determination}]
\label{determination poles}\index{determination of poles of zeta function for $n=2$}
Let $f\in \C[x,y]\setminus \C$, satisfying $f(0)=0$. Consider the {\em minimal} embedded resolution $h$ of the germ of $\{f=0\}$ at $0$, where we use notation from Note \ref{notation embedded resolution}.  Then $s_0$ is a pole of $Z_{\top,0}(f;s)$ if and only if $s_0=-\frac{\nu_j}{N_j}$, with $E_j$ either a component of the strict transform or an exceptional component intersecting at least three times other components.
\end{theorem}

Note that Example \ref{extop} is consistent with this theorem (as it should).
So the \lq false\rq\ candidate poles arise {\em only} from exceptional components intersecting once or twice other components.

\begin{remark}
 There is a certain partial resolution of $f$, where exactly those components  do not appear, the so-called {\em relative log canonical model\index{relative log canonical model}}, a concept from the Minimal Model Program.
Here we just note that this partial resolution space has in general some mild singularities, more precisely cyclic quotient singularities.
 A compact formula for $Z_{\top,0}(f;s)$ (as well as for the $p$-adic zeta function), in terms of this partial resolution, where thus all appearing candidate poles are effective poles, was derived in \cite{Veys lcmodel}. Such a formula for the motivic zeta function, in a more general twodimensional setting, was treated in  \cite{RV2}.

This work has been generalized to arbitrary dimension: compact formulas,  both in topological and motivic setting, were derived in terms of certain partial resolutions where the ambient space has abelian quotient singularities, see \cite{LMVV}. These are used in for instance work on the monodromy conjecture in a nondegenerate setting \cite{Q}, mentioned in \S\ref{special polynomials}.
\end{remark}

\section{Higher dimension}
\label{sec:4}

\subsection{Strategy for surfaces ($n=3$)}\label{sec:surfaces}

Recall that, for $n=2$, one could expect that (generically)  an exceptional curve $E_j$ does not contribute to the poles of $Z_{\top}(f;s)$ or $Z_{\top,0}(f;s)$ as soon as $\chi(E_j^{\circ}) \geq 0$.
We proved this in Proposition \ref{non contribution}, which was an important step in the proof of the monodromy conjecture for curves.

For $n=3$, the corresponding expectation is that (generically) an exceptional surface $E_j$ does not contribute to the poles of $Z_{\top}(f;s)$ or $Z_{\top,0}(f;s)$ as soon as $\chi(E_j^{\circ}) \leq 0$.
We proved this expectation almost in complete generality \cite{Veys pv integrals}; in this subsection we outline the main ideas in the proof, including some new geometric results.
(We formulate this outline for $Z_{\top}(f;s)$, but it is also valid for $Z_{\top,0}(f;s)$.)

\smallskip
Let $f\in \C[x,y,z]\setminus \C$. We fix an embedded resolution $h:Y\to \A_\C^3$ of $\{f=0\}$, for which we use notation from Note \ref{notation embedded resolution}.  Further, we may and will assume that $h$ is constructed as a composition of admissible blow-ups as in \cite{Hironaka}, where each centre of blow-up is contained in the consecutive strict transform of $\{f=0\}$.
 We focus on exceptional surfaces\index{exceptional surface} $E_j$ that are mapped to a point by $h$, and hence are projective (in the case of isolated singularities we may assume this for all exceptional $E_j$).

When is $\chi(E_j^{\circ}) \leq 0$?  In order to investigate this, it is important to realize the difference with the case of curves. When $n=2$, an exceptional curve $E_j \subset Y$ is created as a $\P^1$ by blowing up some point during $h$, and its strict transform by the further blow-ups in $h$ stays isomorphic to $\P^1$. Furthermore, the number of intersection points of that exceptional curve with other components does not change during the resolution process.  When $n=3$, an exceptional surface $E_j\subset Y$ is either created, by blowing up a point, as  $E_j^{*} \cong \P^2$\index{projective plane}, or, by blowing up a nonsingular curve $Z$, as a ruled surface\index{ruled surface} $E_j^{*}$ over $Z$ (by which we mean a $\P^1$-bundle over $Z$). During further blow-ups of $h$, this $E_j^{*}$ often \lq changes\rq.  Indeed, assume that the next centre of blow-up is a point $P$ in $E_j^{*}$, or a curve intersecting $E_j^{*}$ transversally in $P_1,\dots,P_m$. Then the strict transform of $E_j^{*}$ after this blow-up is isomorphic to $E_j^{*}$, blown up in $P$, or in $P_1,\dots,P_m$, respectively. The remaining case, a blow-up with centre a curve contained in  $E_j^{*}$, does not change  $E_j^{*}$. Finally, at the end of the resolution process, $E_j\subset Y$ is isomorphic to the projective plane or  ruled surface $E_j^{*}$, blown up  finitely many (say $r$) times.
Moreover, then $E_j$ has $r$ more intersecting (irreducible) curves with other components than $E_j^*$ at the moment of its creation.

However, this change does not influence the \lq open part\rq\ $E_j^\circ$ of that exceptional surface $E_j$. Indeed, let $C$ denote the union of all intersections of $E_j^{*}$ with other components after creation. Then, in the process of constructing $E_j$ out of $E_j^{*}$ by point blow-ups, all changes occur within $C$ and its further consecutive total transform in that exceptional surface.  As a conclusion,
$$E_j^{*}\setminus C   \cong  E_j \setminus \{\text{intersection of } E_j \text{ with other components in } Y\} = E_j^\circ,$$
and hence $\chi(E_j^{*}\setminus C) = \chi (E_j^\circ)$.

When the surface $E_j$  (or, equivalently, $E_j^*$) is not rational, it turns out that the condition $\chi(E_j^{\circ}) \leq 0$ is quite special, and that in these cases $E_j$ does not contribute to the poles of $Z_{\top}(f;s)$, see \cite{Veys BullSMF} and \cite{Veys pv integrals}.
But already in the case $E_j^* \cong \P^2$, there is a whole zoo of configurations $C \subset \P^2$ with $\chi(\P^2 \setminus C)= 0$, and an even larger zoo with  $\chi(\P^2 \setminus C)\leq 0$. For instance, there exist such $C$ with any possible number of irreducible components, and of arbitrarily high degree and multiplicity of the singularities of $C$.

\begin{example} As an illustration, we present two examples of such $C=\cup_\ell C_\ell \subset \P_\C^2$ with $\chi(\P^2 \setminus C)= 0$.

(1) Let $C_0=\{y^kz - x^{k+1}=0\}$, where $k \geqslant 2$;   $C_1 = \{y = 0\}$;  $C_2=\{ z = 0\}$;  $C_3 =\{x= 0\}$.

(2) Let $C_0 =\{ \prod^s_{i=1} (y - a_ix)^{m_i} - x^{k+1}=0\}$, where $s \geqslant 1$, $m_i \geqslant 1$ for $1\leq i \leq s$, and     $k = \sum^s_{i=1} m_i \,\,(\geqslant
2)$;    $C_i = \{y - a_i x=0\}$  ($1 \leqslant i \leqslant s$);   $C_j =\{ y - b_jx=0\}$  ($1 \leqslant j \leqslant t$), where $t \geqslant 0$ and the numbers $a_i$ and $b_j$ are all different.

\bigskip
\centerline{
\beginpicture
\setcoordinatesystem units <.3truecm,.3truecm>

\putrectangle corners at -1 0 and 12 11
\putrule from 1 7 to 9 7
\putrule from 8 1 to 8 10
\plotsymbolspacing=.2pt
\setlinear    \plot  2 9   9 2  /
\plotsymbolspacing=.3pt
\circulararc 90 degrees from 4 7 center at 4 10
\circulararc -120 degrees from 4 7 center at 4 3
\plotsymbolspacing=.4pt

\put {$C_0$} at 6.7 1.1
\put {$C_1$} at 10 7
\put {$C_2$} at 8.8 9.6
\put {$C_3$} at 1.4 9.6

\put {(1)} at 5.5 -1.4

\setcoordinatesystem units <.3truecm,.3truecm> point at -15 0

\putrectangle corners at 0 0 and 14 11
\putrule from 1 5 to 13 5
%\savelinesandcurves on "save3-3"
\plotsymbolspacing=.3pt
\setlinear    \plot  4 1  8 9  /
              \plot  8 1  4 9  /
              \plot  2 6.3333  12 3 /
              \plot  2 7.6666  12 1   /

\plotsymbolspacing=.4pt
\ellipticalarc axes ratio 3:1 180 degrees from 9 4 center at 6 4
\startrotation by -.447 -.854  about 6 5
\ellipticalarc axes ratio 3:1  -90 degrees from 6 5 center at 6 4
\stoprotation
\startrotation by -.447  .854
\ellipticalarc axes ratio 3:1 90 degrees from 6 5 center at 6 4
\stoprotation

\setquadratic    \plot   7 9   6.75 7   6 5  /
                  \plot   5 9   5.25 7   6 5  /

\circulararc 60.435 degrees from 3 4 center at 4.383 4.15
\circulararc -60.435 degrees from 9 4 center at 7.617 4.15

\put {$C_0$} at 2.7 2.7
\put {$C_i$} at 13 5.6
\put {$C_j$} at 1.3 7

\put {(2)} at 7 -1.4

\endpicture  }
\end{example}

\medskip
For many more examples of such families, and all possible cases where the degree of $C$ is at most $4$, see \cite{Veys BullSMF}.
In that paper we also formulated the following as a conjecture. It was first shown by de Jong and Steenbrink in \cite{dJS}; see also \cite{GP} and \cite{Kojima} for other proofs.

\begin{theorem}\label{rational}\index{Euler characteristic}\index{curve complements in the projective plane}\index{projective plane}
If $C$ is a curve in $\P_\C^2$ with $\chi(\P^2 \setminus  C) \leq 0$, then all irreducible components of $C$ are rational.
\end{theorem}

Anyway, a classification of such configurations seems impossible. Instead, we discovered the following \lq structure theorem\rq\ \cite{Veys Structure}, that can be applied to $(X,D)=(E_j, \cup_{i\neq j} (E_i\cap E_j))$.

 Recall that a $(-1)$-curve on a nonsingular surface is a nonsingular rational curve with self-intersection $-1$ on that surface.

\begin{theorem}\label{wafels}\index{structure theorem for open surfaces}\index{Euler characteristic}\index{ruled surface}\index{projective plane}
 Let $X$ be a (complex) nonsingular projective rational surface. Let $D$ be a {\em connected} simple normal  crossings curve on $X$ satisfying $\chi(X \setminus D) \leq 0$. Assume that $X$ does not contain any $(-1)$-curve disjoint from $D$.
By \cite[Theorem 3]{GP},
 there exists a dominant morphism $\phi : X \setminus D \to \P^1$; let $h : \tilde X \to  X$ be the {\em minimal} morphism that resolves the indeterminacies of $\phi$, considered as rational map from $X$ to $\P^1$.

\medskip

\centerline{
\beginpicture
\setcoordinatesystem units <.5truecm,.5truecm>

\put{$\searrow$} at 1 1
\put{$\swarrow$} at 1 -1
\put{$\searrow$} at -1 -1
\put{$\swarrow$} at -1 1
\put{$\Bigg\downarrow$} at 0 0

\put{$g$} at -1.3 1.3
 \put{$h$} at 1.3 1.3
 \put{$\tilde \phi$} at -.4 0
 \put{$\phi$} at 1.3 -1.3
 \put{$\pi$} at -1.3 -1.3

\put{$\Sigma$} at -2 0
 \put{$\tilde X$} at  0 2
 \put{$X$} at 2 0
 \put{${\P}^1$} at  0 -2

\endpicture  }

\bigskip
(1) Then there exists a connected curve $D' \supset D$ with  $\chi (X \setminus D') \leq \chi(X \setminus D) \leq 0$, such that the morphism $\tilde \phi=\phi \circ h$ decomposes as
$\tilde X \overset g \to  \Sigma \overset \pi \to  \P^1$,
where $g$ is a composition of blowing-downs with exceptional curve in $h^{-1}D'$, and $\pi : \Sigma \to \P^1$ is a ruled surface. Moreover, $h^{-1}D'$ has simple normal crossings in
$\tilde X$.

\smallskip
(2) We can require the configuration $g(h^{-1}D') \subset \Sigma$ to be one of the configurations below.
Here $C_1$ and $C_2$ are sections of $\pi$, $C$ is a nonsingular curve for which $\pi|_C : C \to \P^1$ has degree 2 (a \lq bisection\rq), and the other curves are fibres of $\pi$. The minimal number of fibres in (a) and (b) is 2 and 1, respectively; in (c) there must pass a fibre through each ramification point of $\pi|_C$, and we can have any number of other fibres. Note that in (c) the bisection can be non--rational (and then has more than two ramification points).

\medskip
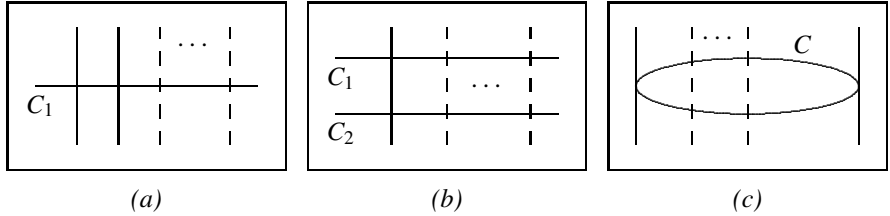
\centerline{
\beginpicture
\setcoordinatesystem units <.37truecm,.37truecm>

\putrectangle corners at 0 0 and 10 6
\putrule from 1 3 to 9 3
\putrule from 2.5 0.9 to 2.5 5.1
\putrule from 4 0.9 to 4 5.1
\setdashes
\putrule from 5.5 0.9 to 5.5 5.1
\putrule from 8 0.9 to 8 5.1

\setsolid

\put {$\dots$} at 6.75 4.5
\put {$C_1$} at 1.2 2.3

\put {(a)} at 5 -1.1

\setcoordinatesystem units <.37truecm,.37truecm> point at -10.8 0

\putrectangle corners at 0 0 and 10 6
\putrule from 1 2 to 9 2
\putrule from 1 4 to 9 4
\putrule from 3 0.9 to 3 5.1
\setdashes
\putrule from 5 0.9 to 5 5.1
\putrule from 8 0.9 to 8 5.1
\setsolid

\put {$\dots$} at 6.5 3
\put {$C_1$} at 1.2 3.3
\put {$C_2$} at 1.2 1.3

\put {(b)} at 5 -1.1

\setcoordinatesystem units <.37truecm,.37truecm> point at -21.6 0

\putrectangle corners at 0 0 and 10 6
\putrule from 1 0.9 to 1 5.1
\putrule from 9 0.9 to 9 5.1
\setdashes
\putrule from 3 0.9 to 3 5.1
\putrule from 5 0.9 to 5 5.1
\setsolid

\ellipticalarc axes ratio 4:1  360 degrees from 1 3 center at 5 3

\put {$\dots$} at 4 4.75
\put {$C$} at 7 4.5

\put {(c)} at 5 -1.1

\endpicture  }
\medskip
(3) When $D\subset X$ is obtained from a curve on  $\P^2$ by a composition of blow-ups, then the end configuration (c) is not possible.
\end{theorem}

\begin{example}\label{wafelvoorbeeld}
Let $C_0 =\{y^2z - x^3=0\}$ and  $C'_0 = \{y = 0\}$ in $\P^2$. The minimal embedded resolution $X\to \P^2$ of $C_0\cup C'_0$ is obtained by three  blow-ups, with consecutive exceptional curves $C_1,C_2,C_3$. Consider  $D=C_1 \cup C_2 \cup C_3\cup C_0\cup C'_0 \subset X$ as data in Theorem \ref{wafels}; note that $\chi(X\setminus D) = \chi (\P^2 \setminus (C_0\cup C'_0))=0$. Blowing down first $C'_0$ and then $C_2$ yields a ruled surface $\Sigma$ with configuration $g(D)$ as in case (b).
In the figure, the numbers between square brackets denote the self-intersection numbers of the curves.

\bigskip

\centerline{
\beginpicture

\setcoordinatesystem units <.5truecm,.5truecm>

\putrectangle corners at -2 -3 and 4 3
\putrule from -1 0 to 3 0
\ellipticalarc axes ratio 2:3  70 degrees from 0 0 center at 0 3
\ellipticalarc axes ratio 2:3  -70 degrees from 0 0 center at 0 -3

\put {$\bullet$} at 0 0
 \put {$C_0 [9]$} at 2.8 -1.5
 \put {$C'_0 [1]$} at 3 .5
\put{$\P^2$} at 4.6 1.8

\setcoordinatesystem units <.37truecm,.37truecm> point at 20 4

\putrectangle corners at 0 0 and 10 8
\put {$\bullet$} at 3 5.5
\putrule from 1 2 to 9 2
\putrule from 1 4 to 9 4
\putrule from 3 0.9 to 3 6.5

\put {$C_1 [-3]$} at 8 3.3
\put {$C_0 [3]$} at 8 1.3
\put {$C_3 [0]$} at 4.4 7
\put {$\Sigma$} at -.8 6

\setcoordinatesystem units <.37truecm,.37truecm> point at 13 -7

\putrectangle corners at 0 0 and 12 9
\putrule from 1 2 to 9 2
\putrule from 1 4 to 9 4
\putrule from 3 0.9 to 3 7.2
\putrule from 1.5 6 to 8 6
\putrule from 7 5 to 7 8

\put {$C_1 [-3]$} at 8 3.3
\put {$C_0 [3]$} at 8 1.3
\put {$C_3 [-1]$} at 3.5 7.7
\put {$C_2 [-2]$} at 9.5  5.8
 \put {$C'_0 [-1]$} at 8.6 8
\put {$X=\tilde X$} at -1.5 7.8

\put{$\searrow$} at 10 -1.5
\put{$\swarrow$} at 2 -1.5
 \put {$g$} at 2.5 -2

\endpicture}

\bigskip
\end{example}

\begin{remark}  Theorems \ref{rational} and \ref{wafels} are nice examples of novel geometric results of independent interest, inspired by the quest for a proof of the monodromy conjecture.
\end{remark}

We return now to our fixed embedded resolution $h:Y\to \A_\C^3$ of $\{f=0\}$, and an exceptional surface $E_j$, that is created during $h$ either as a projective plane or as a (projective) rational ruled surface $E_j^*$.
Let $E_i, i\in S_j (\subset S),$ be the other components that intersect $E_j$. We assume that $E_j$ induces $-\frac {\nu_j}{N_j}$ as candidate pole of order $1$, that is,  $\frac {\nu_j}{N_j}\neq \frac {\nu_i}{N_i}$ for all $i\in S_j$.

In order to use Theorem \ref{wafels} with $X=E_j$, we need $\cup_{i\in S_j}( E_j \cap E_i)$ to be connected.  This is automatic when $E_j^*\cong\P^2$ and turns out to be \lq almost always\rq\ true when $E_j^*$ is a ruled surface. Indeed, otherwise for instance {\em no fibres} of the ruled surface $E_j^*$ are intersections with other components, which is very rare in a resolution process.

With Theorem \ref{wafels} as main input, and using the relations between numerical data from \cite{Veys Relations} and some subtle arguments,  we obtained the following very general vanishing result for candidate poles of order $1$ \cite[3.7]{Veys pv integrals}.

\begin{theorem}\index{non-contribution to poles} Let  $E_j$ be a rational projective exceptional surface,  inducing $-\frac{\nu_j}{N_j}$ as candidate pole of order $1$ of $Z_{\top}(f;s)$, and satisfying $\chi(E_j^\circ) \leq 0$.

(1) If $E_j$ is created by blowing up a point, then $E_j$ does not contribute to the poles of $Z_{\top}(f;s)$.

(2) If $E_j$ is created by blowing up a (projective) rational curve and $\cup_{i\in S_j}( E_j \cap E_i)$ is connected, then $E_j$ does not contribute to the poles of $Z_{\top}(f;s)$.
\end{theorem}

We illustrate with the above example how Theorem \ref{wafels} can be used to show that the contribution of $E_j$ to the residue of the candidate pole $-\frac{\nu_j}{N_j}$ is zero.

\begin{example} (continuing Example \ref{wafelvoorbeeld})
Let  $E_j$ and its intersection with other components be as in Example \ref{wafelvoorbeeld}, that is, $X=E_j$ is created as  $E_j^*\cong\P^2$, the intersections with other components at that stage of the resolution process are $C_0$ and $C'_0$, and the intersections of $E_j$ with other components in the final embedded resolution are $C_0$, $C'_0$, $C_1$, $C_2$ and  $C_3$, whose union is the curve $D$.
Say those five curves are more precisely the intersections of $E_j$ with the surfaces $E_0$, $E'_0$, $E_1$, $E_2$ and  $E_3$, respectively. Denote $\alpha_0:= \nu_0 -\frac{\nu_j}{N_j}N_0, \dots, \alpha_3:=\nu_3 -\frac{\nu_j}{N_j}N_3$.

The contribution of $E_j$ to the residue of  $-\frac{\nu_j}{N_j}$ is the following expression $\mathcal{R}_{E_j}$ in terms of Euler characteristics and these $\alpha$'s (compare with the residue expression  (\ref{residue}) in the curve case; here we  have 10 terms, coming from the natural stratification of $X=E_j$ induced by the normal crossings curve $D$):
$$
\mathcal{R}_{E_j} =
0+\frac1{\alpha_0}+\frac1{\alpha'_0}+\frac1{\alpha_1}+\frac0{\alpha_2}+\frac{-1}{\alpha_3}
+\frac1{\alpha'_0 \alpha_2}+\frac1{\alpha_2 \alpha_3}+\frac1{\alpha_1 \alpha_3}+\frac1{\alpha_0 \alpha_3} .
$$
From \cite{Veys Relations}, we have the basic relation $3\alpha_0 + \alpha'_0 -1=0$, associated to the original configuration on $E_j^*\cong\P^2$, and the additional relations
$$\alpha_1=2\alpha_0 + \alpha'_0 -1, \quad \alpha_2=\alpha_0 + \alpha'_0 +\alpha_1-1 \quad\text{and}\quad \alpha_3=\alpha_0 + \alpha_1+\alpha_2 -1,$$
 associated to each blow-up leading to $E_j$.

In a concrete case, as here,  plugging in such relations in the expression $\mathcal{R}_{E_j}$ would yield zero when expected.  But this does not allow proving a general theorem.
We illustrate further the strategy that does allow it.  One can associate a similar expression to any normal crossings curve $\cup_i C_i$ on a smooth projective surface $A$, where each $C_i$ is equipped with a nonzero rational number $\alpha_i$, satisfying $K_A=\sum_i(\alpha_i-1)C_i$. For instance in the example we have such an expression  $\mathcal{R}$ for the intermediate surface obtained after contracting $C'_0$ in $E_j$ and an expression $\mathcal{R}_{\Sigma}$ for $\Sigma$.

Now the relations between the $\alpha$'s above yield in particular that
$$\alpha'_0=\alpha_2 +1  \qquad\text{and} \qquad \alpha_2=\alpha_3+1,$$
and it is easy to verify that such \lq blow down relations\rq\ are exactly what is needed to obtain that $\mathcal{R}_{E_j} =\mathcal{R} =\mathcal{R}_{\Sigma}$.
The expression $\mathcal{R}_{\Sigma}$ is
$$
0 + \frac1{\alpha_0}+ \frac1{\alpha_1}+ \frac0{\alpha_3} + \frac1{\alpha_0 \alpha_3}+\frac1{\alpha_1 \alpha_3}=\Big(\frac1{\alpha_0}+ \frac1{\alpha_1}\Big)\Big(1+\frac1{\alpha_3}\Big),
$$
and this is zero because
$$\alpha_0 + \alpha_1=0.$$
This last relation also follows from the relations between the $\alpha$'s above.  An important point in general is that we already knew that such a relation must hold. It follows from the fact that $K_\Sigma=(\alpha_0-1)C_0 + (\alpha_1-1)C_1 + (\alpha_3-1)C_3$ and the adjunction formula, applied to a generic fibre of $\Sigma$.

\end{example}

\subsection{Special polynomials}\label{special polynomials}

Here we present an overview of the many special families of polynomials, for which the monodromy conjecture was proven, mostly for the standard version, sometimes for the more difficult stronger version featuring the b-function.  Early results concern only the $p$-adic zeta function, but are mostly also valid in the context of  the motivic and (a fortiori) the topological zeta function.
Later results are often stated in the  setting of the motivic or topological zeta function.

In most cases, proofs in one setting (topological, $p$-adic, motivic) imply or can be upgraded to proofs in the other settings.  But in some particular cases a proof in the $p$-adic or motivic setting is really more difficult than for the topological zeta function (typically when complicated combinatorial formulas are involved).

\smallskip

Everything started with Igusa's calculations of the $p$-adic zeta function for polynomials $f$ that are relative invariants of prehomogeneous vector spaces\index{prehomogeneous vector space}.  (We refer to e.g. \cite{SS} for these notions. Igusa uses the symmetry of the associated group; resolution of singularities does not appear here.)   This study can be reduced to 29 types; for most of them  Igusa derived in several papers explicit expressions for the zeta function, showing that the real parts of its poles are roots of $b_f$. This is summarized in  \cite{Ig}.  Kimura, Sato and  Zhu treated all types using microlocal analysis   \cite{KSZ}.

\smallskip
For surface singularities ($n=3$), the standard version is proved for several families of polynomials. Rodrigues and Veys \cite{RV} showed it for \lq almost all\rq\ homogeneous polynomials\index{homogeneous polynomial}. Artal, Cassou-Nogu\`es, Melle and Luengo proved it for all  superisolated singularities\index{superisolated singularity}, and also for the remaining case of homogeneous polynomials  \cite{ACLM1}.  Polynomials that are general with respect to a threedimensional toric idealistic cluster\index{toric idealistic cluster} were handled by Lemahieu and Veys \cite{LV}.

As a side remark, we mention the exploration of a possible generalization of the monodromy conjecture to functions on normal surface singularities \cite{Ro}\cite{RV2}.

\smallskip
Polynomials that are  nondegenerate with respect to their local or global Newton polyhedron\index{Newton nondegenerate polynomial}\index{nondegenerate polynomial} attracted considerable attention. All proofs use (mutually related) combinatorial formulas for the respective zeta functions: $p$-adic  \cite[Theorem 4.2]{DenefHoornaert},  topological    \cite[Th\'eor\`eme 5.3]{DenefLoeser1}, and motivic \cite[Theorem 10.5]{BV}.
(Another such formula, in the context of the generalizations in the next section, is derived in  \cite{Bo1}.)

Loeser showed an instance of the strong version of the conjecture for arbitrary $n$, assuming however several technical conditions \cite{Lo3}. There are many contributions to the standard version.
For $n=3$, it was shown unconditionally by Lemahieu and Van Proeyen  \cite{LVP} for the topological zeta function, and by Bories  and Veys for the $p$-adic and motivic zeta function  \cite{BV}. (This is a case where proofs in the $p$-adic and motivic setting are considerably more difficult.)
More recently Esterov, Lemahieu and Takeuchi  \cite{ELT} introduced new arguments for both existence of eigenvalues and cancellation
of candidate  poles for the topological zeta function, especially for $n = 4$, establishing in particular the conjecture unconditionally for $n=4$.

In arbitrary dimension,  Larson,  Payne and  Stapledon proved the conjecture for the  motivic zeta function, assuming the nondegenerate polynomials have {\em simplicial} Newton polyhedra \cite{LPS}.
Also in arbitrary dimension, Quek showed  cancellation of candidate poles for the motivic zeta function, by constructing a stack-theoretic embedded resolution  \cite{Q}. This yields in particular a new geometric proof of the main result of \cite{BV}.

\smallskip
Finally, we list families of polynomials in an arbitrary number of variables $n$, for which the conjecture is proved.

Artal, Cassou-Nogu\`es, Melle and Luengo showed the standard version for quasi-ordinary polynomials\index{quasi-ordinary polynomial} on the motivic level \cite{ACLM2}.

For  hyperplane arrangements\index{hyperplane arrangement}, Budur, Musta\c t\v a and Teitler showed the standard version for the topological zeta function  \cite{BMT}. In that paper, they also reduce the stronger version to showing that, for an indecomposable essential central hyperplane arrangement of degree d, we have that $-n/d$ is a root of its b-function. The latter is shown in some cases by Budur, Saito and Yuzvinski in \cite{BSYV}, featuring also an appendix by Veys containing unexpected examples where that value $-n/d$ is {\em not} a pole of the topological zeta function.

Blanco, Budur and van der Veer proved the strong version for semi-quasi\-homo\-geneous polynomials\index{semi-quasihomogeneous polynomial} with an {\em isolated singularity} at the origin  \cite{BBvdV}.

Budur and van der Veer showed the standard and strong version for polynomials $fg$, where $g$ is a  {\em log generic}\index{log generic polynomial} and {\em log very-generic}\index{log very-generic polynomial} polynomial with respect to $f$, respectively
\cite{BvdV}.

\section{Generalizations}
\label{sec:5}

For ease of presentation, and in order to focus on the \lq essentials\rq, we exposed everything in the most basic setting.  In this final section we present the original formulation of the conjecture, in the context of more general $p$-adic fields and also involving a multiplicative character.  We also mention subsequent generalizations: mappings/ideals instead of only one polynomial, zeta functions involving a non-standard differential form, and zeta functions in several variables $s_\ell$.

One can combine several generalizations simultaneously, but here we present each generalizing aspect separately in order to focus on that aspect.

\subsection{Original formulation}

Let $K$ be an arbitrary $p$-adic field\index{$p$-adic field}, that is, a finite extension of $\Q_p$. Let $\mathcal{O}_K$ be its valuation ring with maximal ideal $m_K$, generated by a uniformizing parameter $\pi$, and residue field $\mathcal{O}_K/m_K$ of cardinality $q$ (which is a power of $p$).

Any $z\in K^\times$ can be written uniquely in the form $z=\pi^{\ord_K(z)} \cdot \ac(z)$, where $\ord_K(z)\in \Z$ is the {\em $\pi$-order} of $z$ and $\ac(z)\in \mathcal{O}_K^\times$ is the {\em angular component} of $z$.
The standard {\em norm} of $z$ is $|z|_K:= q^{-\ord_K(z)}$.

\begin{definition}
Let $K$ be a $p$-adic field as above, $f \in K [x_1, \dots, x_n] \setminus K$ and $\varphi$ a  locally constant function $K^n\to \C$ with compact support. Let furthermore $\varkappa:\mathcal{O}_K^\times \to \C^\times$ be a (multiplicative) character\index{character}\index{multiplicative character}, that is, a group homomorphism with finite image. To these data one associates the {\em $p$-adic Igusa zeta function}\index{$p$-adic zeta function}\index{Igusa zeta function}\index{$p$-adic Igusa zeta function}
$$
Z_K(f,\varkappa,\varphi;s) := \int_{K^n} |f(x)|_K^s \varkappa\big(\ac f(x)\big) \varphi(x) dx ,
$$
where $dx$ denotes the standard Haar measure on $K^n$ (such that $\mathcal{O}_K^n$ has measure $1$), and $s\in \C$ with $\Re(s)>0$.
\end{definition}

\begin{remark}
This zeta function can be defined more conceptually without choosing a uniformizing parameter and the associated angular component. A {\em quasi-character}\index{quasi-character} of $K^\times$ is a continuous homomorphism $\omega:K^\times \to \C^\times$ (extended with $\omega(0)=0$), to which one associates the zeta function $Z_K(f,\varphi,\omega):= \int_{K^n} \omega(f(x)) \varphi(x) dx$.
Choosing a uniformizing parameter $\pi$ and  $s \in \C$ satisfying $\omega(\pi) = q^{-s}$, we have that $\omega (z) =  |z|_K^s \varkappa(\ac z)$,  where $\varkappa := \omega|_{\mathcal{O}_K^\times}$.
 Then $Z_K(f,\varphi,\omega) = Z_K(f,\varkappa,\varphi;s)$.
\end{remark}

Generalizing Theorem \ref{zetafunction is rational}, $Z_K(f,\varkappa,\varphi;s)$ is a rational function in $q^{-s}$, and thus has a meromorphic continuation to $\C$.
Also Denef's formula\index{Denef's formula} (Theorem \ref{denefformula}) generalizes, see \cite[Theorem 2.2]{Denef2}.  Here we just mention the adapted list of candidate poles.

\begin{proposition}\label{poles with character}
 Let $h:Y\rightarrow \A_K^n$ be an embedded resolution of $\{f=0\}$, for which we use  notation from Note \ref{notation embedded resolution}. Then, denoting by $d$ the order of the character $\varkappa$, the poles of $Z_K(f,\varkappa,\varphi;s)$  are of the form
$$
 -\frac{\nu_j}{N_j} + \frac{2\pi k}{(\ln p)N_j}i, \text{ with } \ j\in S \text{ such that } d \mid N_j \text{, and } k\in \Z \, .
$$
\end{proposition}

The original formulation of the monodromy conjecture is in this setting.

\begin{conjecture}[Monodromy conjecture]\index{monodromy conjecture}
Let $F$ be an arbitrary number field and $f\in F[x_1,\dots,x_n]\setminus F$.  Then, for all but a finite number of (non-archimedean) completions $K$ of $F$, all characters $\varkappa$ of $\mathcal{O}_K^\times$ and all test functions $\varphi$ on $K^n$,
 if $s_0$ is a pole of $Z_K(f,\varkappa, \varphi;s)$, then $e^{2\pi i \Re(s_0)}$ is a monodromy eigenvalue of $f:\C^n\to \C$ at some point of $\{f=0\}$.
\end{conjecture}

There is an analogous local version and a stronger version with the b-function $b_f$.  The general form of the topological zeta function is as follows, where the number $d$ arises as the order of the character $\varkappa$ in the \lq limit procedure\rq\ of \S\ref{subsec:topological zeta function}.  %\ref{sec:2}

\begin{definition}[\cite{DenefLoeser1}]
Let $f\in \C[x_1,\dots,x_n]\setminus \C$ and $d \in \Z_{\geq 1}$. Let $h:Y\rightarrow \A_\C^n$ be an embedded resolution of $\{f=0\}$, for which we use  notation from Note \ref{notation embedded resolution}.
The (global) {\em topological zeta function\index{topological zeta function} associated to $f$ and $d$} is
$$
Z_{\top}(f,d;s) :=	\sum_{I\subset S, \forall i\in I: d|N_i}\chi(E_I^\circ)\prod_{i\in I}\frac{1}{\nu_i+ N_is}.
$$
	\end{definition}

Also for this variant one can generalize the monodromy conjecture\index{monodromy conjecture} in the obvious way.
For the motivic generalization, and the formulation of the monodromy conjecture in that setting, we refer to \cite[2.4]{DenefLoeser2}. Note that, with respect to this more general motivic zeta function\index{motivic zeta function}, the \lq basic\rq\ motivic zeta function from \S\ref{motivic} is sometimes called the {\em naive motivic zeta function}.

For an upgrade of the proof for $n=2$ to this setting, see \cite[\S8]{BN}.

\medskip
According to Proposition \ref{poles with character},  when $d$ divides no $N_j$ at all, then $Z_K(f,\varkappa,\varphi;s)$ is holomorphic on $\C$. These $N_j$ are not intrinsically associated to $f$, but the order (as root of unity) of any
monodromy eigenvalue of $f$ divides some $N_j$, and those eigenvalues are intrinsic invariants of $f$.  This observation inspired Denef  to propose the following  \cite{Denef3}.

\begin{conjecture}[Holomorphy conjecture]\index{holomorphy conjecture}
Let $F$ be an arbitrary number field and $f\in F[x_1,\dots,x_n]\setminus F$.  Then, for all but a finite number of (non-archimedean) completions $K$ of $F$, we have the following. If the order of a character $\varkappa$ of $\mathcal{O}_K^\times$ does not divide the order of any monodromy eigenvalue of $f$ at any complex point of $\{f=0\}$, then $Z_K(f,\varkappa,\varphi;s)$ is holomorphic on $\C$.
\end{conjecture}

We proved this for curves in \cite{Veys Holom curves}; other results are in \cite{DV}\cite{RV}\cite{Veys Euler}\cite{CIL}.

\subsection{Mappings/ideals}

For simplicity we take as test function $\varphi$ the standard choice: the characteristic function of $\Z_p^n$.

\begin{definition}\index{zeta function for mapping}
Let $f_1,\dots,f_r  \in\Q_p[x_1,\dots,x_n]$,  such that the ideal generated by  $f_1,\dots,f_r$ is not trivial.
The {\em $p$-adic Igusa zeta function\index{$p$-adic zeta function}\index{Igusa zeta function}\index{$p$-adic Igusa zeta function} associated to the mapping $\mathbf{f}=(f_1,\dots,f_r):\Q_p^n\to \Q_p^r$}  is %and $\varphi$ is
$$
Z(\mathbf{f};s) := \int_{\Z_p^n} ||\mathbf{f}(x)||_p^s %\varphi(x)
dx ,
$$
where $||\mathbf{f}(x)||_p = ||(f_1(x),\dots,f_r(x))||_p := \max\{|f_1(x)|_p,\dots,|f_r(x)|_p\}$.
\end{definition}

Many aspects of the one polynomial case generalize in a natural way. Fixing $p$ and taking all $f_j$ in $\Z[x_1,\dots,x_n]\setminus \Z$, knowing  $Z(\mathbf{f};s)$ is equivalent to the knowledge of the numbers of solutions of  $f_1=\ldots=f_r=0$ over the finite rings $\Z/p^i\Z$.
There is a similar \lq Denef type\rq\ formula\index{Denef's formula}, in terms of a principalization of the ideal generated by $f_1,\dots,f_r$ in $\Q_p[x_1,\dots,x_n]$. When the $f_j$ are defined over $\C$ or over any field $K$ of characteristic zero, we have a similar notion of topological\index{topological zeta function} and motivic zeta function\index{motivic zeta function}, respectively, associated to the ideal\index{zeta function for ideal} $I=(f_1,\dots,f_r)$ \cite{Veys-Zuniga1}.

On the other hand, there is no clear generalization of the notion of local Milnor fibre of a function $\C^n\to \C$ to an arbitrary polynomial map $\C^n\to\C^r$. But there is a notion of \lq Verdier monodromy\index{Verdier monodromy}\rq\ \cite{Ver}, associated to any ideal $I$ in $\C[x_1,\dots,x_n]$  (and even in a more general setting of subschemes of schemes), and also a notion of b-function\index{b-function} $b_I(s)$ for such an ideal \cite{BMS}, both generalizing the one polynomial case.  Moreover, a generalization of Proposition \ref{Malgrange} holds in this setting: if $s_0$ is a root of $b_I(s)$, then $e^{2\pi i s_0}$ is a Verdier monodromy eigenvalue of $I$, and every Verdier monodromy eigenvalue is obtained this way  \cite{BMS}.  We initiated the question whether the various versions of the monodromy conjecture for one polynomial could generalize to this setting of ideals, that is, do poles of zeta functions of ideals $I$ induce roots of $b_I(s)$ or Verdier monodromy eigenvalues of $I$?\index{monodromy conjecture}

The strong version was verified for monomial ideals in \cite{HMY}.  The (Verdier) monodromy version was  proven for arbitrary ideals in two variables in \cite{VPV2}, and for all families of space monomial curves with a plane semigroup in \cite{MVV} (see also \cite{MMVV}).
Of independent interest is an A'Campo type formula\index{A'Campo's formula}  for \lq Verdier monodromy zeta functions\rq\ in arbitrary dimension \cite[Theorem 3.2]{VPV2}.

A variant of the holomorphy conjecture\index{holomorphy conjecture} in this setting was proved for $n=2$ \cite{LVP2}.

\medskip
Most importantly, Musta\c t\v a \cite{Mu2} related $b_I(s)$ to the b-function of {\em one} polynomial (in more variables), and similarly for the zeta function $Z(\mathbf{f};s)$.

\begin{theorem}[\cite{Mu2}] (1) Let $I=(f_1,\dots,f_r)$ be a nontrivial ideal in $\C[x_1,\dots,x_n]$.  Consider the polynomial $g:= \sum_{i=1}^r f_i y_i\in \C[x_1,\dots,x_n,y_1,\dots,y_r]$.
Then $b_g(s) = (s+1)b_I(s) $.

(2) Let $f_1,\dots,f_r \in \Q_p[x_1,\dots,x_n]$ generate a nontrivial ideal.  Consider the polynomial $g:= \sum_{i=1}^r f_i y_i\in \Q_p[x_1,\dots,x_n,y_1,\dots,y_r]$.
Then
$$
Z(g;s) = \frac{1-p^{-1}}{1-p^{-s-1}} Z(\mathbf{f};s) .
$$
\end{theorem}

As a corollary, the monodromy conjecture for one polynomial {\em in all dimensions} implies the monodromy conjecture for ideals in all dimensions.

\subsection{Zeta functions with differential form}

Recall Corollary \ref{poles induce eigenvalues} in the context of the real zeta functions of \S\ref{archimedean}: if $s_0$ is a pole of $Z(f,\varphi;s)$, then $e^{2\pi i s_0}$ is a monodromy eigenvalue of $f$ at some point of $\{f=0\}$.  This is also true for complex zeta functions associated to a complex polynomial; and in that case also the converse holds: {\em all} monodromy eigenvalues of $f$ are obtained this way \cite{Ba}. (There are similar converse results for real polynomials and zeta functions by Barlet and collaborators.)

We now shift our point of view;  for a test function $\varphi:\R^n\to \C$, we consider $\varphi dx$ together as the measure associated to the  $C^\infty$ differential $n$-form $\varphi dx =\varphi(x)dx_1\wedge\dots\wedge dx_n$.  Then a converse result as above says that a given monodromy eigenvalue of $f$ can be obtained from a pole of an archimedean zeta function associated to $f$ and a form $\varphi dx$.

Is there some analogue of this more general result in the $p$-adic case?  For sure, locally constant functions are much less flexible than $C^\infty$ functions;
typically \lq  most\rq\  monodromy eigenvalues of $f$  are {\em not} induced by poles of $p$-adic zeta functions.
We proposed the following setting, taking here for simplicity as test function the characteristic function on $(p\Z_p)^n$.

We view the Haar measure $dx$ on $\Q_p^n$ as the measure associated to the standard differential $n$-form $dx=dx_1\wedge\dots\wedge dx_n$, and consider the more general $n$-form $\omega=gdx$, with $g\in \Q_p[x_1,\dots,x_n]$, inducing the measure given by $|g(x)|_p dx$.

\begin{definition}\index{zeta function for polynomial and differential form}
Let $f,g \in \Q_p [x_1, \dots, x_n]$ with $f\notin\Q_p$ and $g\neq 0$.  The {\em $p$-adic Igusa zeta function\index{$p$-adic zeta function}\index{Igusa zeta function}\index{$p$-adic Igusa zeta function} associated to $f$ and $\omega=gdx$} is
$$
Z_0(f,\omega;s) := \int_{(p\Z_p)^n} |f(x)|_p^s |g(x)|_p dx.
$$
\end{definition}

There is a Denef type formula\index{Denef's formula} for $Z_0(f,\omega;s)$, generalizing Theorem \ref{denefformula}, now in terms of an embedded resolution of $\{fg=0\}$, leading to a topological zeta function also in this context.

\begin{definition}\index{zeta function for polynomial and differential form}
Let $f,g\in \C[x_1,\dots,x_n]$ with $f\notin \C$ and $g\neq 0$. Let $h:Y\rightarrow \A_\C^n$ be an embedded resolution of $\{fg=0\}$, for which we adapt somewhat the notation from Note \ref{notation embedded resolution}. That is, now
$\cup_{j\in S}E_j = h^{-1}\{fg=0\}$,  $N_j$ is still the multiplicity of $E_j$ in $\div(f\circ h)$, but now $\nu_j -1$ is the multiplicity of $E_j$ in $\div(h^*(gdx))$.

The {\em  local topological zeta function\index{topological zeta function} associated to $f$ and $\omega = gdx$ at $a\in \{f=0\}$} is
\begin{equation}
Z_{\top,a}(f, \omega;s) :=	\sum_{I\subset S}\chi(E_I^\circ \cap h^{-1}\{a\})\prod_{i\in I}\frac{1}{\nu_i+ N_is}.
\end{equation}
	\end{definition}

Now, given a monodromy eigenvalue $\lambda$ of $f$ at $0\in\{f=0\}$, does there exist a form $\omega$ and a pole $s_0$ of $Z_{\top,a}(f, \omega;s)$ (for some $a$ close to $0$) such that $e^{2\pi is_0}=\lambda$?
The answer is yes, in complete generality in all dimensions, see  \cite[Theorem 3.6]{Veys forms}.  When $0$ is an isolated singularity, we can take $a=0$.

However, typically such a $Z_{\top,a}(f, \omega;s)$ has many other poles which {\em do not} induce monodromy eigenvalues of $f$.  We initiated the following question/programme, which is pretty ambitious, since it would be a vast generalization of the classical monodromy conjecture.

\begin{problem}\label{big MC}\index{monodromy conjecture}
Given $f\in \C[x_1,\dots,x_n]\setminus \C$ satisfying $f(0)=0$, identify/define a class of {\em allowed  forms\index{allowed differential form} $\omega=gdx$ for $f$},  such that
\begin{enumerate}
\item[(1)] if $s_0$ is a pole of $Z_{\top,0}(f, \omega;s)$, where $\omega$ is allowed, then $e^{2\pi is_0}$ is a monodromy eigenvalue of $f$ (at a point close to $0$);
\item[(2)] the standard form $dx=dx_1\wedge\dots\wedge dx_n$ is allowed;
\item[(3)] if $\lambda$ is a monodromy eigenvalue of $f$ at $0$, then there exists an allowed form $\omega$ and a pole $s_0$ of  $Z_{\top,0}(f, \omega;s)$ such that $e^{2\pi is_0}=\lambda$.
\end{enumerate}
\end{problem}

One can imagine global versions or slightly adapted local versions. We like to call this problem the {\em generalized monodromy conjecture}.\index{generalized monodromy conjecture}

For $n=2$, we realized this programme in collaboration with N\'emethi \cite{NV1}\cite{NV2}, via the concept of \lq  splicing\rq\index{splicing}.  For a given  $f$, we consider its simplified dual resolution graph $\Gamma_f$, deleting all vertices of valence $2$, hence leaving only vertices of valence at least 3 (called {\em nodes}) and of valence 1 (end vertices). Choosing any edge $e$ connecting two nodes (assuming $\Gamma_f$ contains at least two nodes), we cut $\Gamma_f$ along $e$  and put a new end vertex at each side of the cut, creating this way two new graphs . It was known that these can be interpreted as simplified dual resolution graphs of two polynomials $f_L$ and $f_R$, such that the monodromy zeta function of $f$ is essentially the product of the monodromy zeta functions of $f_L$ and $f_R$. We showed that moreover $Z_{\top,0}(f,\omega;s)$ is (up to a simple correction term) the sum of  $Z_{\top,0}(f_L, \omega_L;s)$ and $Z_{\top,0}(f_R, \omega_R;s)$, for well chosen forms $\omega_L$ and $\omega_R$.  (When $\omega$ is the standard form, $\omega_L$ and $\omega_R$ usually are {\em not}!)
Continuing such splicing, we decompose $\Gamma_f$ into a number of \lq star shaped\rq\ graphs, each with exactly one node.

In the special case that $\Gamma_f$ itself is a star shaped graph, we establish a pretty natural definition of allowed forms for $f$, realizing the programme.  For general $f$, both the definition of allowed form\index{allowed differential form} for $f$ and our proof of the statements in Problem \ref{big MC} are obtained by splicing $\Gamma_f$ as above in star shaped pieces, and analyzing carefully the behaviour of eigenvalues and poles under this operation.
Note that in particular this yields a new proof of the \lq classical\rq\ monodromy conjecture for $n=2$.

\begin{example}(continuing Example \ref{example})  Consider again $\{f=y^3-x^5=0\}$. Its minimal embedded resolution $h$  is also an embedded resolution of $\{xyf=0\}$. This is illustrated in the figure, where $C_1$ and $C_2$ denote $\{x=0\}$ and $\{y=0\}$ and their strict transforms, respectively. (Note that $\Gamma_f$ is star shaped.)

\bigskip

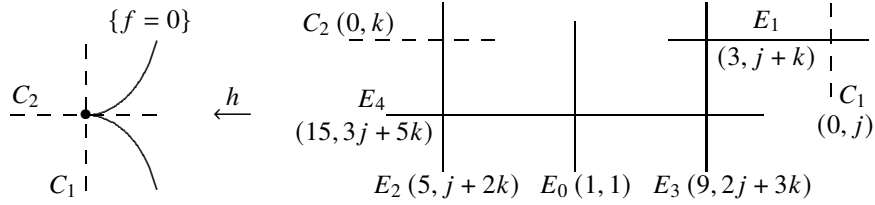
\centerline{
\beginpicture
\setcoordinatesystem units <.5truecm,.5truecm>

 \setdashes
 \putrule from -2 0 to 2 0
 \putrule from 0 -2 to 0 2
 \setsolid

\ellipticalarc axes ratio 2:3  70 degrees from 0 0 center at 0 3
\ellipticalarc axes ratio 2:3  -70 degrees from 0 0 center at 0 -3

\put {$\bullet$} at 0 0
 \put {$ C_1$} at -.6 -1.9
 \put {$ C_2$} at -1.6 .5
 \put {$\{f=0\}$} at 1.7 2.5

\put {$\longleftarrow$} at 3.9 0
 \put {$h$} at 3.9 .5

\setcoordinatesystem units <.5truecm,.5truecm> point at -13 0

 \putrule from -5 0  to 5 0
 \putrule from  0  -1.5 to 0 2.5
 \putrule from -3.5  -1.5  to -3.5 3
 \putrule from 3.5  -1.5  to 3.5 3
 \putrule from 7.8 2 to 2.5 2

\setdashes
 \putrule from -6 2 to -2 2
\putrule from 6.8 .5  to 6.8 3
 \setsolid

 \put {$E_1$} at 5.1 2.5       \put{$(3,j+k)$}  at 5.1 1.5
  \put {$E_2 \, (5,j+2k)$}  at  -3.4  -1.9
  \put {$E_3 \, (9,2j+3k)$}  at  4.2  -1.9
  \put {$E_4$} at -5.4  .4     \put{$(15,3j+5k)$}  at -5.6  -.5
   \put {$E_0\, (1,1)$}  at  .3  -1.9
  \put {$C_2\,  (0,k)$} at -6  2.4
\put{$C_1$} at 7.4 .5   \put{$(0,j)$} at 7.2  -.3

\endpicture}

\bigskip
\noindent
 The forms $\omega_{jk} := x^{j-1}y^{k-1}dx\wedge dy$, with $j\in \{1,2,3,4\}$ and $k\in \{1,2\}$,  are typical examples of allowed forms for $f$, and we claim that they realize the programme above.
We indicate in the figure all numerical data $(N_i,\nu_i)$, where now the $\nu_i$ come from the form $\omega_{jk}$.
Then by a straightforward computation, or by \cite[Proposition 2.7]{Veys forms}, we have that
$$Z_{\top,0}(f, \omega_{jk};s) = \frac{3j+5k + (3j+5k-jk)s}{jk(1+s)(3j+5k+15s)},$$
 hence its poles are precisely $-1$ and $s_{jk}:=-\frac{3j+5k}{15}$.
Consequently the $e^{2\pi i s_{jk} }$ cover exactly the primitive $15$th roots of unity.
\end{example}

The \lq almost additive formula\rq\ for the topological zeta function was upgraded to the motivic zeta function\index{motivic zeta function} by Cauwbergs \cite{Ca}, resulting in realizing the programme\index{generalized monodromy conjecture} also for the motivic zeta function when $n=2$.  (A weaker programme, omitting (2) in Problem \ref{big MC}, was realized in \cite{CV}, using some dirty trick.)

\smallskip
One could wonder about a stronger version of Problem \ref{big MC}, replacing monodromy eigenvalues by roots of the b-function\index{b-function}.  This was investigated by Bories \cite{Bo2} for $n=2$, obtaining some negative answers and a partially positive one. Then a next question is: does there exist some variant of Problem \ref{big MC} (with eigenvalues replaced by roots of the b-function) that could be realistic?

\subsection{Multivariate zeta functions}

There are obvious multivariate generalizations of the archimedean and $p$-adic Igusa zeta functions associated to several polynomials, see e.g. \cite{Lo2}, and of the induced topological and motivic zeta functions.  We mention the basic $p$-adic and topological versions.

\begin{definition}\index{multivariate zeta function}
(1) Let $f_1,\dots,f_k \in \Q_p [x_1, \dots, x_n] \setminus \Q_p$ and put $F=(f_1,\dots,f_k)$. The multivariate {\em $p$-adic Igusa zeta function\index{$p$-adic zeta function}\index{Igusa zeta function}\index{$p$-adic Igusa zeta function} associated to $F$} is
$$
Z(F;\mathbf{s}) := \int_{\Z_p^n} \prod_{\ell=1}^k |f_\ell(x)|_p^{s_\ell}  dx ,
$$
where all $s_\ell \in \C$ with $\Re(s_\ell)>0$, and $\mathbf{s} = (s_1,\dots,s_k)$.%

(2)
Let $f_1,\dots,f_k\in \C[x_1,\dots,x_n]\setminus \C$; we put $F=(f_1,\dots,f_k)$ and $f=\prod_{\ell=1}^k f_\ell$. Choose an embedded resolution $h:Y\rightarrow \A_\C^n$ of $\{f=0\}$, for which we use notation from Note \ref{notation embedded resolution}, with moreover $\div(f_\ell\circ h)=\sum_{j\in S} N_j^{(\ell)} E_j $ for $\ell=1,\dots,k$.
The (global) {\em topological zeta function\index{topological zeta function} of $F$} is
\begin{equation}
Z_{\top}(F;\mathbf{s}) :=	\sum_{I\subset S}\chi(E_I^\circ)\prod_{i\in I}\frac{1}{\nu_i+ N_i^{(1)}s_1+\dots+N_i^{(k)}s_k}.
\end{equation}
	\end{definition}

We focus here on the topological one.    We denote its {\em polar locus}, that is, the support of the divisor of poles, by $\mathcal{P}(Z_{\top}(F;\mathbf{s}))$. It is a union of hyperplanes, among the
candidate polar hyperplanes\index{polar hyperplane} $\{\nu_j+ N_j^{(1)}s_1+\dots+N_j^{(k)}s_k=0\}, j\in S$.

The generalization of the monodromy eigenvalues (when $k=1$) is the monodromy support of $F$, defined
as follows. For each $\alpha\in (\C^*)^{k}$, one  has a rank one local system on $(\C^*)^{k}$ with monodromy $\alpha_i$ around the $i$th missing coordinate hyperplane. Let $\mathcal{L}_\alpha$  denote its  pullback  via $F|_{\C^n\setminus \{f=0\}}: \C^n\setminus \{f=0\}\to (\C^*)^k$. For $x\in \{f=0\}$, let $U_x$ be the complement of $\{f=0\}$ in a small open ball around $x$ in $\C^n$. Then the {\em monodromy support\index{monodromy support} $\mathcal{S}_{F}$ of $F$} consists of those $\alpha\in (\C^*)^{k}$ for which there exists a point $x\in \{f=0\}$ with
$H^*(U_x, \mathcal L _\alpha)\ne 0$.
(It is a finite union of torsion-translated codimension one affine algebraic subtori of $(\C^*)^{k}$ \cite{BLSW}.)

\begin{conjecture}
Let $F$ be a $k$-tuple of nonconstant polynomials in $\C[x_1,\ldots,x_n]$. Then
$$\Exp\Big(\mathcal{P}\big(Z_{\top}(F;\mathbf{s})\big)\Big)\subset \mathcal{S}_F,$$
where  $\Exp:\C^k\to (\C^*)^k$ is the coordinate wise  map $\beta\mapsto e^{2\pi i \beta}$.
\end{conjecture}

For $k> 1$, this generalized conjecture was posed by Loeser, see  \cite{Ni} and \cite{Lo2}.
 It is known for tuples of plane curves \cite{Ni}, and tuples of hyperplane arrangements\index{hyperplane arrangement} \cite{B-ls}.

A stronger version involves a generalization of the Bernstein-Sato polynomial,  the {\em Bernstein-Sato ideal\index{Bernstein-Sato ideal} $B_F$ of $F$}, which is the ideal generated by polynomials $b\in\C[s_1,\dots,s_k]$ satisfying
$$
P\prod_{\ell=1}^k f_\ell^{s_\ell+1} =  b\prod_{\ell=1}^k f_\ell^{s_\ell}
$$
for some $P\in\mathcal{D}_n[s_1,\ldots,s_{k}]$.
Let $Z(B_F)$ denote the zero locus of $B_F$ in $\C^k$. (It was  proven in \cite{BVWZ} that $\Exp(Z(B_F))=\mathcal{S}_F$, extending Theorem \ref{Malgrange}.)
The stronger version then states that $\mathcal{P}(Z_{\top}(F;\mathbf{s}))\subset Z(B_F)$.

This stronger version was posed in \cite{B-ls}. It is known for tuples factorizing a tame hyperplane arrangement \cite{Bath}, and for tuples of linear polynomials \cite{Wu}.
In \cite{BvdV} instances of both standard and stronger version are proved in the context of so-called {\em log generic\index{log generic polynomial}} and {\em log very-generic\index{log very-generic polynomial}} polynomials, respectively.

\bigskip
\noindent
{\bf Acknowledgements.} The author is indebted to Jan Denef for introducing him a long time ago to this fascinating topic in the broad world of singularities. He thanks the organizers of the IberoSing International Workshop (2022) in Madrid for forcing him to prepare a course on the monodromy conjecture, that served as inspiration for this text, and also Guillem Blanco and the referees for their help in improving  the text.

%%%%%%% End of tex

%%%%%%% Bibliography

% BibTeX users please use
% \bibliographystyle{}
% \bibliography{}
% Otherwise, see the file referenc.tex to see how to format the references. It is loaded in the following command:
% \input{referenc}

%\bibliographystyle{amsalpha}
\makeatletter \renewcommand{\@biblabel}[1]{\hfill#1.}\makeatother

%\bibliography{}

%\printindex-
\end{document}